\documentclass[12pt]{amsart}

\usepackage{amsthm}
\newtheorem{theorem}{Theorem}[section]
\newtheorem*{theorem*}{Theorem}
\newtheorem{lemma}[theorem]{Lemma}

\theoremstyle{definition}
\newtheorem{definition}[theorem]{Definition}
\newtheorem{question}[theorem]{Question}

\theoremstyle{remark}
\newtheorem{remark}[theorem]{Remark}
\newtheorem{example}[theorem]{Example}
\newtheorem{corollary}[theorem]{Corollary}

\counterwithin*{section}{part}

\usepackage{amssymb}
\usepackage{empheq}
\usepackage{stmaryrd}
\usepackage{enumerate}
\usepackage[shortlabels]{enumitem}
\usepackage{calc}
\usepackage{url}
\usepackage{comment}

\newcommand{\norm}[1]{\left\lVert#1\right\rVert}
\usepackage{mathtools}

\newcommand{\CAT}{{\rm{CAT}(0)}}

\DeclareMathOperator{\dsum}{\overrightarrow{\sum}}

\def\EE{\mathbb{E}}
\def\PP{\mathbb{P}}

\usepackage[left=2.0cm,%
                right=2.0cm,%
                top=2.5cm,%
                bottom=3.5cm,%
                headheight=12pt,%
                a4paper]{geometry}%

\begin{document}

\title[Strong laws, random monotone vector fields and gradient flows]{Strong laws, random monotone vector fields and gradient flows on metric spaces of nonpositive curvature}

\author[N. Pischke]{Nicholas Pischke}
\date{\today}
\maketitle
\vspace*{-5mm}
\begin{center}
{\scriptsize 
Department of Computer Science, University of Bath,\\
Claverton Down, Bath, BA2 7AY, United Kingdom.\\
E-mail: nnp39@bath.ac.uk}
\end{center}

\maketitle
\begin{abstract}
We establish a probabilistic Lie-Trotter-Kato formula for the gradient flow generated by an integral function over suitable separable Hadamard spaces. This result leverages Stojkovi\'c's Lie-Trotter-Kato formula for gradient flows of sums of convex functions and Ba\v{c}\'ak's result relating resolvent and gradient flow convergence, along with a strong law of large numbers for normal convex integrands. For separable Hilbert spaces and locally compact Hadamard spaces, these strong laws follow from results of King and Wets and Artstein and Wets, respectively, but we crucially go beyond these examples by establishing a new such strong law for normal convex integrands on all Hadamard spaces whose tangent cones are isometric to Hilbert spaces, a class that includes non-Hilbertian, non-locally compact examples such as infinite-dimensional hyperbolic space. This new strong law arises as a combination of a strong law of large numbers for integrable random monotone vector fields on Hilbert-Hadamard spaces, that is, Hadamard spaces whose tangent cones all embed isometrically into Hilbert spaces, extending a result of Salim from Hilbert spaces, together with a new result on the interchangeability of the subdifferential and the integral, which we establish for $L^2$-Lipschitz integrands. This latter ingredient provides, to our knowledge, the first nonlinear version of (a particular case of) a previous result due variously to Ioffe and Tikhomirov, Levin, Hiriart-Urruty, Thibault as well as Rockafellar and Wets. On the way to this strong law for random monotone vector fields, we also establish a distribution-uniform generalization of Sturm's strong law of large numbers for inductive means of $L^1$-sequences of i.i.d.\ random variables on Hadamard spaces via a new effective non-asymptotic concentration inequality. Throughout the paper, we highlight various remaining open problems.
\end{abstract}


\section{Introduction}

The strong law of large numbers, originally established by Kolmogorov \cite{Kolmogorov1930}, states that for independent and identically distributed (i.i.d.)\ integrable random variables $(X_i)$ on a probability space $(\Omega,\mathsf{F},\PP)$, their sample averages converge to their common mean with probability one, i.e. 
\[
\frac{1}{n}\sum_{i=1}^n X_i\to \EE[X_1]\quad \PP\text{-a.s.\ as }n\to\infty. 
\]
Due to its ubiquity both in theoretical probability and statistics as well as their applications, this result belongs arguably to the most important results in all of probability theory, and many extensions and variants of the strong law of large numbers have been studied over the years (see e.g.\ the surveys \cite{Revesz1968,RosalskyThanh2025}, as well as \cite{Neri2025b,Neri2025a,Neri2026} for recent quantitative studies of various related results).

In the present paper, we focus on strong laws of large numbers that take place over complete geodesic metric spaces of nonpositive curvature, so-called Hadamard spaces. Introduced by Alexandrov  \cite{Aleksandrov1951,Alexandrov1957} and often called $\CAT$ spaces since the influential work of Gromov \cite{Gromov1987}, this class unifies spaces with smooth structure such as (infinite-dimensional) Hilbert spaces and Hadamard manifolds (i.e.\ complete, simply connected Riemannian manifolds of nonpositive sectional curvature) with (potentially discrete) structures such as $\mathbb{R}$-trees or the Billera-Holmes-Vogtmann tree space prominently used in phylogenetics \cite{BilleraHolmesVogtmann2001}, as well as many more. We generally refer to \cite{AlexanderKapovitchPetrunin2019,AlexanderKapovitchPetrunin2023,Bacak2014a,BridsonHaefliger1999} for various treatments of these spaces, which are a key common ground for many analytic, geometric and probabilistic results, with techniques from either area being used to study the others. Probability theory in particular provides various fundamental tools by which analytical or geometrical aspects of these spaces can be studied, going back to the work of Jost \cite{Jost1994,Jost1997} and in particular Sturm \cite{Sturm2001,Sturm2002b,Sturm2003}. 

The present paper now aims to illustrate how such strong laws over nonpositively curved spaces can be used in combination with other results from stochastic convex analysis to derive new stochastic approximation principles for deterministic gradient flows generated by integral functions. Concretely, as is well-known, if $\mathcal{H}$ is a Hadamard space and $f:\mathcal{H}\to (-\infty,+\infty]$ is of the form $f=\sum_{k=1}^Nf_k$ such that each $f_k:\mathcal{H}\to (-\infty,+\infty]$ is a convex and lower semicontinuous function, then the gradient flow semigroup $S_t(x)$ associated with $f$ (see \cite{AmbrosioGigliSavare2008}) can be characterized by a limit of products of the individual proximal mappings $\mathrm{prox}_{\lambda f_1}$, \dots, $\mathrm{prox}_{\lambda f_N}$ of $f_1$, \dots, $f_N$, as opposed to a limit of the proximal mapping $\mathrm{prox}_{\lambda f}$ of $f$ as in the usual Crandall-Liggett-type formula (see \cite{Mayer1998}). More precisely, one has the following limit for all $t\geq 0$ and all $x\in \overline{\mathrm{dom}f}$:
\[
S_t(x)=\lim_{n\to\infty} \left(\mathrm{prox}_{\frac{t}{n} f_N} \circ\dots\circ \mathrm{prox}_{\frac{t}{n} f_1}\right)^{(n)}(x).
\]
This product formula provides a nonlinear analogue of the well-known Lie-Trotter-Kato product formula for such gradient flow semigroups in linear spaces (see \cite{KatoMasuda1978}, among many other related works) and was first established by Stojkovi\'c \cite{Stojkovic2012} (and later simplified by Ba\v{c}\'ak \cite{Bacak2014c}).\footnote{We also mention the work of Ohta and P\'alfia \cite{OhtaPalfia2017} which extended this result to $\mathrm{CAT}(1)$-spaces. Generally, while not addressed explicitly throughout, a natural question for future work is whether the results of the present paper extend to more general upper curvature bounds.}

The ultimate goal of the present paper is to establish the following probabilistic variant of the Lie-Trotter-Kato formula (somewhat similar in spirit to \cite{Kurtz1972}) for the gradient flow generated by an integral function $\int f(e,\cdot)\, d\mu(e)$ for a proper normal convex integrand $f:E\times \mathcal{H}\to (-\infty,+\infty]$, relating the flow of the integral to powers of products of random samples:

\begin{theorem*}[informal, Theorem \ref{probLTK} in the paper]
Let $(\Omega,\mathsf{F},\PP)$ and $(E,\mathsf{E},\mu)$ be suitable probability spaces and let $(\mathcal{H},d)$ be a suitable separable Hadamard space. Let $f:E\times \mathcal{H}\to (-\infty,+\infty]$ be a proper normal convex integrand which is $L^2$-Lipschitz and such that $f(\cdot,u_0)$ is integrable for some $u_0\in \mathcal{H}$. Write $S_t(x)$ for the gradient flow associated with $\int f(e,\cdot)\, d\mu(e)$, and fix an i.i.d.\ sequence $(\xi_i)$ of random variables $\xi_i:\Omega\to E$ with distribution $\mu$.

Then, there exists a set $\Omega'\subseteq\Omega$ with $\PP(\Omega')=1$ such that for any $\omega\in\Omega'$ as well as for any $t\geq 0$ and $x\in \mathcal{H}$:
\[
S_t(x)=\lim_{n\to\infty} \left(\mathrm{prox}_{\frac{t}{n^2} f(\xi_n(\omega),\cdot)} \circ\dots\circ \mathrm{prox}_{\frac{t}{n^2} f(\xi_1(\omega),\cdot)}\right)^{(n)}(x).
\]
In particular, the above limit holds almost surely.
\end{theorem*}

The probability space $(E,\mathsf{E},\mu)$ later has to satisfy completeness and separability restrictions, and the class of permissible Hadamard spaces is characterized by the condition that their tangent cones are isometric to Hilbert spaces (so-called full Hilbert-Hadamard spaces, which encompass Hilbert spaces as well as manifolds with nonpositive sectional curvature but also go beyond these classes, as discussed in more detail later).

The proof of the above result makes crucial use of the usual Lie-Trotter-Kato formula for sums of convex functions already mentioned previously and in particular a new uniform quantitative variant thereof tailored to Lipschitz functions (see Lemma \ref{quantLTK} later), as well as of a result relating resolvent and gradient flow convergence established by Ba\v{c}\'ak \cite{Bacak2015}. 

However, as mentioned before, the main novel ingredients in the respective argument are strong laws of large numbers together with further results from stochastic convex analysis over these spaces. Concretely, there is the following sequence of intermediate results of independent interest leading up to the above theorem:

\begin{enumerate}[(A)]
\item We derive a (distribution-uniform) strong law of large numbers for inductive means of i.i.d.\ sequences of integrable random variables over general Hadamard spaces.
\item We utilize the strong law for integrable random variables to derive a strong law of large numbers for i.i.d.\ selections of integrable random monotone vector fields on so-called Hilbert-Hadamard spaces (characterized by a flatness condition on the tangent cones).
\item To apply the strong law for integrable random monotone vector fields in the context of convex functions, we establish a result on the interchange between the subdifferential and the Aumann integral on full Hilbert-Hadamard spaces (characterized by a linearity condition on the tangent cones, as mentioned before).
\end{enumerate}

The use of this sequence of results can be avoided in certain cases. In the case of Hilbert spaces (see Theorem \ref{probLTKHilbert} later on), the results can be replaced by an application of a strong law for convex functions due to King and Wets \cite{KingWets1991}, while in the case of a locally compact Hadamard space (see Theorem \ref{probLTKProper} later on, which in particular does not require any further assumptions on the tangent cones), one can utilize a related strong law due to Artstein and Wets \cite{ArtsteinWets1995}. In both cases, the $L^2$-Lipschitz assumption can actually be relaxed to $L^1$-Lipschitz continuity. Hence, the succession of results (A) -- (B) -- (C) is only required to establish the above theorem for Hadamard spaces (satisfying the respective tangent cone restrictions) which are neither linear nor locally compact, such as e.g.\ the infinite dimensional hyperbolic space $\mathbb{H}^\infty$ (see \cite{Gromov1993}).

Both of the strong laws in (A) and (B) are stronger than actually required in the later proof of the probabilistic Lie-Trotter-Kato formula, and both strong laws are actually established by means of a quantitative result which entails the various uniformities. On the way to each of these results and the theorem above, we prove various other related results of independent interest, and develop a considerable apparatus of stochastic nonlinear analysis for monotone vector fields over Hadamard spaces. We use this coherent and expansive picture not aimed exclusively at a single result to formulate a range of remaining open problems in the area.

The following sections of the introduction now describe each of the above ingredients and related results in more detail.

\subsection*{A. Strong laws for inductive means of random variables on Hadamard spaces}

The strong law of large numbers for bounded random variables on Hadamard spaces goes back to the work of Sturm \cite{Sturm2003} (see Theorem 4.7 therein, and also Theorem 7.2.1 in \cite{Bacak2014a}): Over a probability space $(\Omega,\mathsf{F},\PP)$ and a separable Hadamard space $(\mathcal{H},d)$, if $(X_i)$ is an i.i.d.\ sequence of random variables $X_i\in L^\infty(\Omega,\mathcal{H},\PP)$, then 
\[
\frac{1}{n}\underset{i=1,\dots,n}{\dsum} X_i\to \EE[X_1]\quad \PP\text{-a.s.\ as }n\to\infty. 
\]
Here, $\frac{1}{n}{\dsum}_{i=1,\dots,n} X_i$ refers to an iteratively defined notion of sample average, the so-called inductive mean introduced by Sturm (see Definition \ref{IndMean} later). These inductive means, which coincide with the usual Ces\`aro means over the reals, are a crucial feature of this strong law since they allow one to actually compute the sample average without prior knowledge or ability to compute the mean of the sequence. This is in particular opposed to other notions of approximate means, such as e.g.\ those induced by barycenters (see e.g.\ \cite{EsSahibHeinich1999}), in which case the strong law would be much easier to prove but also of much less interest as the iteration itself becomes already hard to compute (see Remark 4.8 in \cite{Sturm2003} for further discussions). This strong law has various applications, and can be used to develop further probability theory on these spaces (such as e.g.\ for deriving Jensen's inequality \cite{Sturm2003}) or for developing algorithms to approximate Fr\'echet means in practice (see e.g.\ \cite{Bacak2014b}).

Many variants of this result have been studied since then (see \cite{BrunelSerres2024,ChoiHeoJi2017,Funano2010,KostenbergerStark2025,LekaPalfia2024,LimPalfia2020,OhtaPalfia2015,Pischke2025,PischkePowell2026,TeranMolchanov2006,Yokota2018}, among others), but to our knowledge only the work of Yokota \cite{Yokota2018} explicitly provides an extension of Sturm's strong law of large numbers using inductive means to $L^1$-sequences of i.i.d.\ random variables over general Hadamard spaces, mirroring Kolmogorov's original strong law. Indeed, all the other works above (while extending the strong law in various interesting ways in their own right) deviate from this general formulation, either regarding the notion of average by avoiding inductive means, still only entailing a corresponding strong law for bounded or square-integrable variables, or focusing on different (perhaps not even Hadamard) spaces and operator-theoretic formulations.

Our first main result is set in the context of such a general strong law of large numbers on Hadamard spaces using inductive means of $L^1$-sequences of i.i.d.\ random variables, where we establish an even stronger quantitative result by giving the following non-asymptotic concentration inequality:\footnote{These quantitative results, as well as the remaining quantitative results in this paper, have been informed and obtained by the logic-based methodology of proof mining \cite{Kohlenbach2008,Kohlenbach2019} and as such are part of a recent advance of applying these methods in the context of probability theory \cite{NeriOlivaPischke2026,NeriPischke2024}. This paper however avoids any further reference to mathematical logic, as is common in proof mining.}

\begin{theorem*}[Theorem \ref{L1LawHadQuant} in the paper]
Let $(\Omega,\mathsf{F},\PP)$ be a probability space and let $(\mathcal{H},d)$ be a separable Hadamard space. If $(X_i)$ is an i.i.d.\ sequence of random variables $X_i\in L^1(\Omega,\mathcal{H},\PP)$, then for any $\varepsilon>0$ and any $m\in\mathbb{N}$:
\[
\PP\left(\sup_{n\geq m} d\left(\frac{1}{n}\underset{i=1,\dots,n}{\dsum} X_i,\EE[X_1]\right)>\varepsilon\right)\leq \frac{6}{\varepsilon}\left(\frac{\sqrt{\EE[d(o,X_1)]}}{m^{1/4}}+U_\PP(m^{1/8})\right)+\frac{72}{\min\{\varepsilon^2,1\}m^{1/4}},
\]
where $U_\PP(x):=\EE[d(o,X_1)\mathbf{1}_{d(o,X_1)\geq x}]$ for a fixed $o\in \mathcal{H}$.
\end{theorem*}

While quantitative estimates for inductive means were previously obtained by Funano \cite{Funano2010} or Brunel and Serres \cite{BrunelSerres2024} under substantially stronger moment or tail assumptions, such as boundedness or sub-Gaussianity (and, partially, under restrictions on the space), the above result gives the first explicit non-asymptotic estimate (on an arbitrary separable Hadamard space) under the minimal first-moment assumption. In particular, this provides a new proof of the previous $L^1$ strong law of large numbers established by Yokota \cite{Yokota2018} (see Corollary \ref{L1LawHad} later on). Moreover, this concentration inequality is strong enough to also entail a distribution-uniform generalization of the strong law on Hadamard spaces (see Corollary \ref{chungHad} later on), extending a corresponding result of Chung \cite{Chung1951}. 

While not explored in this paper in detail, the present approach to the (quantitative) strong law of large numbers seems moreover rather general and could perhaps be used to establish other related results in the literature, such as an extension of the $L^1$-version of Birkhoff's pointwise ergodic theorem to Hadamard spaces using inductive means (see Question \ref{ergodicQuestion} later on) or establishing a strong law of large numbers for pairwise independent and identically distributed random variables over Hadamard spaces (see Question \ref{pairwiseQuestion} later on), among others, as discussed in more detail later.

\subsection*{B. Strong laws for inductive means of random monotone vector fields on Hilbert-Hadamard spaces}

Our second main result is an extension of this strong law of large numbers to random monotone vector fields. On Hilbert spaces, so-called maximal monotone operators are key objects in the modern theory underlying convex analysis as well as that of differential equations and their semigroups, among others (see e.g.\ \cite{BauschkeCombettes2017}), and so-called monotone vector fields extend this central notion and its theory to Riemannian manifolds using their tangential structure \cite{LiLopezMartinMarquez2009,LiLopezMartinMarquezWang2011,Nemeth1999,DaCruzNetoFerreiraLucambioPerez2000,WangLopezMartinMarquezLi2010}. Recently, these mappings have in particular also been studied over Hadamard spaces (see e.g.\ \cite{ChaipunyaKohsakaKumam2021}), where they make use of the tangent cones of these spaces in the sense of Alexandrov, Berestovskii and Nikolaev \cite{AleksandrovBerestovskiiNikolaev1986}. Random monotone vector fields are, similarly to random monotone operators on Hilbert spaces, stochastically perturbed versions of these mappings which feature in stochastic convex analysis (see e.g.\ \cite{Bianchi2016,BianchiHachem2016,BianchiHachemSalim2019} as well as the recent \cite{Pischke2025} for extensions to the setting of Hadamard spaces).

A strong law of large numbers for an integrable random monotone operator on Hilbert spaces was recently established by Salim \cite{Salim2023}. Concretely, if $\mathcal{M}(\mathcal{H})$ is the collection of all maximally monotone operators on a separable Hilbert space $\mathcal{H}$, it was shown in \cite{Salim2023} that the averages $\frac{1}{n}\sum_{i=1}^n A(\xi_i)$ of a suitably measurable and integrable mapping $A:E\to \mathcal{M}(\mathcal{H})$ generated via an i.i.d.\ sequence of random variables $\xi_i:\Omega\to E$, given two suitable probability spaces $(\Omega,\mathsf{F},\PP)$ and $(E,\mathsf{E},\mu)$, converge a.s.\ to the integral operator $\int A(e)\, d\mu(e)$ (defined using the Aumann integral \cite{Aumann1965}), that is 
\[
\frac{1}{n}\sum_{i=1}^n A(\xi_i)\to^R \int A(e)\, d\mu(e)\quad\PP\text{-a.s.\ as }n\to\infty,
\]
where convergence here is in the sense of $R$-convergence (see e.g.\ \cite{Attouch1979}), that is in terms of the resolvents of the operators. This strong law of Salim is quite different to previous strong laws of large numbers for random sets \cite{CastaingRaynauddeFitte2013,ShapiroXu2007,TaylorInoue1997,Teran2008,ArtsteinHart1981,ArtsteinVitale1975} and more closely related to strong laws for convex functions \cite{ArtsteinWets1995,KingWets1991}, and was illustrated to have applications to empirical risk minimization (see the discussions in \cite{Salim2023}).
 
Using an appropriate notion of inductive mean for sets in a Hadamard space and a corresponding variant of $R$-convergence, we extend this result of Salim (using a somewhat different approach) to integrable random monotone vector fields (generalizing the respective notion from \cite{Salim2023}):

\begin{theorem*}[informal, Theorem \ref{L1LawVector} in the paper]
Let $(\Omega,\mathsf{F},\PP)$ and $(E,\mathsf{E},\mu)$ be probability spaces and let $(\mathcal{H},d)$ be a suitable separable Hadamard space. Further, fix an i.i.d.\ sequence $(\xi_i)$ of random variables $\xi_i:\Omega\to E$ with distribution $\mu$. Let $A:E\to \mathcal{M}(\mathcal{H})$ be a suitably measurable and integrable mapping, where $\mathcal{M}(\mathcal{H})$ refers to the collection of all monotone vector fields on $\mathcal{H}$ with total resolvents. Then
\[
\frac{1}{n}\underset{i=1,\dots,n}{\dsum} A(\xi_i)\to^R \int A(e)\, d\mu(e)\quad\PP\text{-a.s.}
\]
as $n\to\infty$.
\end{theorem*}

The class of Hadamard spaces that this result is phrased over is that of Hilbert-Hadamard spaces as introduced by Gong, Wu and Yu \cite{GongWuYu2021}, that is Hadamard spaces where every tangent cone is flat (which are discussed in detail later, and also featured in recent investigations into stochastic optimization \cite{Pischke2025}). In particular, the resulting strong law recovers the result from \cite{Salim2023} and additionally holds on Hadamard manifolds and their infinite-dimensional generalizations (as well as other spaces, see Example \ref{HilHadEx} later on).

This result and the surrounding discussion relies on various properties of (random) monotone vector fields on (Hilbert-)Hadamard spaces, established here for the first time, which have independent use in (stochastic) convex analysis where these objects are often employed. Also this strong law of large numbers for vector fields can be made quantitative using a respective concentration inequality phrased using the resolvents (see Remark \ref{vectorLawQuant} later on).

\subsection*{C. An interchange result between the subdifferential and the Aumann integral on full Hilbert-Hadamard spaces}

A canonical example of a (random) monotone vector field is the (random) subdifferential $\partial f(t,x)$ of a proper normal convex integrand $f:T\times \mathcal{H}\to (-\infty,+\infty]$ on a complete probability space $(T,\mathsf{T},\tau)$ and a suitable space $\mathcal{H}$ (see \cite{RockafellarWets1998}). Crucial for both the applicability of the associated strong law of large numbers for random monotone vector fields in the context of gradient flows as illustrated in this paper, but also for other results in (stochastic) convex analysis (see in particular \cite{Pischke2025}, where the question for such a result over Hilbert-Hadamard spaces is raised) are interchange results between the random subdifferential and the integral of the form
\[
\int \partial f(t,x)\,d\tau(t)=\partial\left(\int f(t,\cdot)\,d\tau(t)\right)(x)\text{ for all }x\in \mathcal{H},
\]
where the integral on the left-hand side is again the Aumann integral.

Results of this type were famously considered by Ioffe and Tikhomirov \cite{IoffeTikhomirov1969}, Levin \cite{Levin1970}, Hiriart-Urruty \cite{HiriartUrruty1976}, Thibault \cite{Thibault1981} as well as Rockafellar and Wets \cite{RockafellarWets1982}, among others, and have various uses in stochastic optimization and analysis next to the one mentioned above (see in particular the discussions in the above references).

As a third main result, we establish such an interchange when $f$ is $L^2$-Lipschitz and suitably integrable:

\begin{theorem*}[informal, Theorem \ref{secondInterchange} in the paper]
Let $(T,\mathsf{T},\tau)$ be a suitable probability space and let $(\mathcal{H},d)$ be a suitable separable Hadamard space. Let $f:T\times \mathcal{H}\to (-\infty,+\infty]$ be a proper normal convex integrand and suppose that $f$ is $L^2$-Lipschitz and such that $f(\cdot,u_0)$ is integrable for some $u_0\in \mathcal{H}$. Then, for any $x\in \mathcal{H}$:
\[
\int \partial f(t,x)\,d\tau(t)=\partial\left(\int f(t,\cdot)\,d\tau(t)\right)(x).
\]
\end{theorem*}

Next to being complete, the probability space $(T,\mathsf{T},\tau)$ moreover has to be separable for this result. The class of Hadamard spaces that this result is phrased over is characterized by the assumption that all tangent cones of the Hadamard space are not only flat but are actually isometric to Hilbert spaces. This class of spaces, introduced as full Hilbert-Hadamard spaces in the present paper, is a further restriction of the Hilbert-Hadamard spaces discussed previously, also naturally covering Hilbert spaces and Hadamard manifolds and their infinite-dimensional variants, and which is still not limited to those (see again Example \ref{HilHadEx} later on). In particular, the present result seems to be the first such interchange result even in the manifold case. Also, this assumption of a Hilbert space structure on the tangent cones is in a way sharp since it does not extend to plain Hilbert-Hadamard spaces, as illustrated by an example (see Example \ref{ExNoHH} later on). 

As even full Hilbert-Hadamard spaces do not have a single global dual like linear spaces, establishing this result is rather tricky, and relies on utilizing various tools, including conjugates of convex functions on Hadamard spaces motivated by \cite{BergmannHerzogSilvaLouzeiroTenbrinckVidalNunez2021,SilvaLouzeiroBergmannHerzog2022}, measurable selection theory (see \cite{AubinFrankowska2009,CastaingValadier1977}), directional derivatives of convex functions as studied in \cite{DiMarinoGigliPasqualettoSoultanis2021,MovahediBehmardiHosseini2015} (see also \cite{Lytchak2004,Petrunin2007,Plaut2002}), in particular also making use of recent results of Gigli and Nobili \cite{GigliNobili2021} on the study of the tangent cones of the space of square integrable functions over a Hadamard space, and classical results over Hilbert spaces like Hahn-Banach and the Riesz representation theorem.

\subsection*{Outline of the paper}

The paper is organized as follows: Section \ref{sec:prelim} provides the various preliminaries used throughout, such as background on geodesic metric spaces, Hadamard spaces and their tangent cones as well as the special class of Hilbert-Hadamard spaces and basic probability theory over these spaces. In Section \ref{sec:SLLN}, we derive the (quantitative and distribution-uniform) strong law of large numbers for inductive means of integrable random variables on Hadamard spaces and discuss various related problems and questions. Section \ref{sec:randOp} introduces monotone vector fields on Hadamard spaces as well as their random variants, and derives various initial properties required later. In Section \ref{sec:SLLNVec}, we then derive the strong law of large numbers for integrable random monotone vector fields. Section \ref{interchange} is concerned with the interchange result for the integral and the subdifferential and the final Section \ref{sec:LTK} combines the previous work into the probabilistic Lie-Trotter-Kato formula outlined before, and discusses special and alternative cases.

\section{Preliminaries}\label{sec:prelim}

\subsection{Geodesics, Hadamard spaces and tangent cones}\label{sec:tangCAT}

Let $(\mathcal{H},d)$ be a metric space. A geodesic is an isometry $\gamma:[0,l]\to \mathcal{H}$. We say that $\gamma$ joins $\gamma(0)$ and $\gamma(l)$, as well as that $\gamma$ issues from $\gamma(0)$. Note that necessarily $l=d(\gamma(0),\gamma(l))$. $\mathcal{H}$ is called (uniquely) geodesic if every two points $x,y\in \mathcal{H}$ are joined by a (unique) geodesic and if such a geodesic is unique, we denote it by $\gamma_{x,y}$. A $\CAT$ space (also called a space of nonpositive curvature in the sense of Alexandrov) is a geodesic metric space $(\mathcal{H},d)$ such that the (extended) Bruhat-Tits $\mathrm{CN}$-inequality \cite{BruhatTits1972} holds (see also Proposition 2.3 in \cite{Sturm2003}):
\[
d^2(\gamma(tl),x)\leq (1-t)d^2(\gamma(0),x)+td^2(\gamma(l),x)-t(1-t)d^2(\gamma(0),\gamma(l))
\]
for all $x\in \mathcal{H}$, $t\in [0,1]$ and all geodesics $\gamma:[0,l]\to \mathcal{H}$. Any $\CAT$ space is uniquely geodesic and a complete $\CAT$ space is called a Hadamard space. In that context, we also sometimes write $(1-t)x\oplus t y$ given $x,y\in \mathcal{H}$ and $t\in [0,1]$ for the point $\gamma(tl)$ on the unique geodesic $\gamma:[0,l]\to \mathcal{H}$ with $\gamma(0)=x$ and $\gamma(l)=y$. It should further be noted that every $\CAT$ space is a Busemann space (see e.g.\ Theorem 1.3.3 in \cite{Bacak2014a}), that is it holds that 
\[
d(\gamma(tl),\eta(tk))\leq (1-t)d(\gamma(0),\eta(0))+td(\gamma(l),\eta(k))
\]
for any two geodesics $\gamma:[0,l]\to \mathcal{H}$ and $\eta:[0,k]\to \mathcal{H}$ and any $t\in [0,1]$. As mentioned before, we refer to \cite{AlexanderKapovitchPetrunin2023,Bacak2014a,BridsonHaefliger1999} for comprehensive overviews of geodesic metric spaces, $\CAT$ spaces and Hadamard spaces, including alternative definitions.

For nonconstant geodesics $\gamma$ and $\eta$ issuing from a point $x\in \mathcal{H}$ in a Hadamard space $(\mathcal{H},d)$, we define their Aleksandrov angle $\angle_x(\gamma,\eta)$ by
\[
\angle_x(\gamma,\eta):=\lim_{s,t\to 0^+}\bar{\angle}_x(\gamma(s),\eta(t)),
\]
with $\bar{\angle}_x(y,z)$ referring to the comparison angle, defined via the comparison triangle $\bar{\Delta}(\bar{x},\bar{y},\bar{z})$ of the geodesic triangle $\Delta(x,y,z)\subseteq \mathcal{H}$, as usual (see e.g.\ \cite{BridsonHaefliger1999}). Using the Alexandrov angle, we can define the tangent cone $T_x\mathcal{H}$ of $\mathcal{H}$ at a point $x\in \mathcal{H}$ in the sense of Alexandrov, Berestovskii and Nikolaev \cite{AleksandrovBerestovskiiNikolaev1986}. For that, first note that the Aleksandrov angle $\angle_x$ is a pseudometric on the set of all nonconstant geodesics issuing from $x$. Writing $\Sigma'_x\mathcal{H}$ for the space of all equivalence classes of such geodesics under the equivalence relation defined by $\angle_x(\gamma,\eta)= 0$, the completion $(\Sigma_x\mathcal{H},\angle_x)$ of the space $(\Sigma'_x\mathcal{H},\angle_x)$ is called the space of directions from $x$. The tangent cone $T_x\mathcal{H}$ of $\mathcal{H}$ at $x$ is then the Euclidean cone over $\Sigma_x\mathcal{H}$, that is $T_x\mathcal{H}:=(\Sigma_x\mathcal{H}\times [0,\infty) )/\sim$ where $(\gamma,t)\sim(\eta,s)$ if, and only if, $t=s=0$ or $t=s>0$ and $\gamma=\eta$. For brevity, we write $t\gamma$ for the equivalence class $[(\gamma,t)]_\sim$ and given $u=t\gamma$ and $\lambda\geq 0$, we write $\lambda u:=(\lambda t)\gamma$. On $T_x\mathcal{H}$, we consider the metric
\[
d_x(t\gamma,s\eta):=\sqrt{t^2+s^2-2ts\cos\angle_x(\gamma,\eta)}.
\]
By a result due to Nikolaev \cite{Nikolaev1995}, $T_x\mathcal{H}$ is a Hadamard space itself. We call the set $T\mathcal{H}=\bigcup_{x\in \mathcal{H}}T_x\mathcal{H}$ the tangent bundle of $\mathcal{H}$. On $T_x\mathcal{H}$, we now write $0_x:=0\gamma$ as well as $\norm{t\gamma}_x:=d_x(0_x,t\gamma)=t$ and
\[
g_x(t\gamma,s\eta):=\frac{1}{2}\left(\norm{t\gamma}_x^2+\norm{s\eta}_x^2-d_x^2(t\gamma,s\eta)\right)=ts\cos\angle_x(\gamma,\eta).
\]
Crucially, we have that $g_x(t\gamma,s\eta)\leq \norm{t\gamma}_x\norm{s\eta}_x$, as well as $g_x(t\gamma,t\gamma)=\norm{t\gamma}_x^2$, $g_x(t\gamma,s\eta)=g_x(s\eta,t\gamma)$ and $g_x(t\gamma,s\eta)=tg_x(\gamma,s\eta)$. We refer to  \cite{LewisLopezAcedoNicolae2024} as well as  \cite{AlexanderKapovitchPetrunin2023,BridsonHaefliger1999} for further exposition and proofs of the above results.

\subsection{Flatness and Hilbert-Hadamard spaces}

We call a Hadamard space $\mathcal{H}$ flat (see e.g.\ \cite{KhatibzadehRanjbar2017}) if the $\mathrm{CN}$-inequality is actually an equality, i.e.\ if
\[
d^2(\gamma(tl),x)= (1-t)d^2(\gamma(0),x)+td^2(\gamma(l),x)-t(1-t)d^2(\gamma(0),\gamma(l))
\]
for all $x\in \mathcal{H}$, $t\in [0,1]$ and all geodesics $\gamma:[0,l]\to \mathcal{H}$. Following Gong, Wu and Yu \cite{GongWuYu2021}, we call a Hadamard space $\mathcal{H}$ a Hilbert-Hadamard space if every tangent cone $T_x\mathcal{H}$ is flat.

\begin{remark}\label{flatRem}
This formulation diverges slightly from the definition used in \cite{GongWuYu2021}, where instead of flatness, it is required that $T_x\mathcal{H}$ isometrically embeds into a Hilbert space (which in turn is equivalent to being isometric to a closed and convex subset of a Hilbert space). However, this condition is equivalent to flatness (see in particular Remark 2.4 in \cite{Pischke2025}).
\end{remark}

As discussed in Example 3.2 of \cite{GongWuYu2021}, Hilbert spaces and Hadamard manifolds are Hilbert-Hadamard spaces, as are complete simply connected Hilbert manifolds of nonpositive sectional curvature, that is generalizations of Hadamard manifolds modelled on (potentially infinite-dimensional) Hilbert spaces (such as e.g.\ the infinite dimensional hyperbolic space $\mathbb{H}^\infty$, see \cite{Gromov1993}). Further, Hilbert-Hadamard spaces are closed under taking closed and convex subsets. As a flat space is isometric to a closed and convex subset of a Hilbert space (recall Remark \ref{flatRem}), every flat space is itself a Hilbert-Hadamard space. Therefore, tangent cones of Hilbert-Hadamard spaces are themselves Hilbert-Hadamard spaces. Further, whenever $\mathcal{H}$ is a Hilbert-Hadamard space, so is the space of all square-integrable functions $L^2(T,\mathcal{H},\tau)$ for a finite measure space $(T,\mathsf{T},\tau)$, discussed in detail later. This latter construction in particular allows Hilbert-Hadamard spaces to reach beyond the setting of manifolds (see the discussion in \cite{GongWuYu2021}). We here provide another example of a Hilbert-Hadamard space which is not a smooth manifold:

\begin{example}\label{HilHadEx}
Let $\mathcal{H}_1=\{(x,y)\in\mathbb{H}\mid x\geq 0\}$ be the  right half of the usual Poincar\'e upper half plane $\mathbb{H}$ with constant curvature $-1$ (see \cite{BridsonHaefliger1999}), which is a closed and convex subset of $\mathbb{H}$ and hence a Hadamard space. Let $\mathcal{H}_2$ be a scaled version of $\mathcal{H}_1$ with constant curvature $-2$. As their left boundaries $\{(0,y)\mid y>0\}$ are isometric to $\mathbb{R}$, we can glue $\mathcal{H}_1$ and $\mathcal{H}_2$ along these boundaries to produce $\mathcal{H}:=\mathcal{H}_1\sqcup_\mathbb{R}\mathcal{H}_2$. By Reshetnyak's gluing theorem (see e.g.\ Theorem 11.1 in \cite{BridsonHaefliger1999}), as $\mathbb{R}$ is convex and $\mathcal{H}_1$, $\mathcal{H}_2$ are Hadamard spaces, we get that $\mathcal{H}$ is also a Hadamard space. If $x\in \mathcal{H}$ is away from the seam, it is easy to see that we have $T_x \mathcal{H}\cong\mathbb{R}^2$. If $x\in \mathcal{H}$ is on the seam, then $\Sigma_x\mathcal{H}$ is the gluing of two semi-circles, that is $\Sigma_x\mathcal{H}\cong S^1$. Therefore $T_x\mathcal{H}\cong\mathbb{R}^2$ as well. In particular, $\mathcal{H}$ is a Hilbert-Hadamard space. However, $\mathcal{H}$ cannot be isometric to a smooth Riemannian manifold as the curvature is discontinuous across the seam.
\end{example}

Note that the tangent cones in the above Example \ref{HilHadEx} are not only all flat but actually all isometric to a Hilbert space, which as highlighted in the introduction is a class of Hilbert-Hadamard spaces that will be important for our interchange result between the subdifferential and the integral as well as our probabilistic Lie-Trotter-Kato formula later on.

Whenever $T_x\mathcal{H}$ is flat, we in particular have that $g_x(u,v)$ is affine, that is it is both convex and concave, in both arguments. This follows from a similar argument as in the proof of Theorem 2.4, (4) in \cite{LeGouicParisRigolletStromme2022} and was already explicitly noted in \cite{Pischke2025}.

\begin{lemma}[Lemma 2.6 in \cite{Pischke2025}]\label{affine}
Let $x\in \mathcal{H}$. If $T_x\mathcal{H}$ is flat, then $g_x(u,v)$ is affine in both arguments.
\end{lemma}

Following \cite{KleinerLeeb1997} (see also already \cite{AleksandrovBerestovskiiNikolaev1986,Berestovskii1975}), we define a metric analog $\log_x:\mathcal{H}\to T_x\mathcal{H}$ of the inverse exponential map from Riemannian geometry by $\log_x a:=d(x,a)\gamma_{x,a}$ for $a\neq x$ as well as $\log_x x:=0_x$. Crucially, note that $\log_x$ is nonexpansive:

\begin{lemma}[see e.g.\ p.\ 224 in \cite{GongWuYu2021}]
For any $x,a,b\in \mathcal{H}$:
\[
d_x(\log_xa,\log_xb)\leq d(a,b).
\]
\end{lemma}

The most important property of the pseudo-Riemannian metric $g_x$ on $T_x\mathcal{H}$ is the following estimate relative to the inverse exponential map:

\begin{lemma}[essentially Proposition 2.16 in \cite{ChaipunyaKohsakaKumam2021}]\label{tangentCat0}
For any $x,a,b\in \mathcal{H}$:
\[
g_x(t\log_x a,s\log_x b)\geq \frac{ts}{2}(d^2(x,a)+d^2(x,b)-d^2(a,b)).
\]
\end{lemma}

Further, we will need that $g_x$ is Lipschitz continuous in both arguments.

\begin{lemma}[Lemma 2.3 in \cite{Pischke2025}]\label{nonexpCat0}
For $x\in \mathcal{H}$ and $u,v,w\in T_x\mathcal{H}$:
\[
\vert g_x(u,v)-g_x(u,w)\vert\leq \norm{u}_xd_x(v,w).
\]
\end{lemma}

\subsection{Measurability and integration on $\CAT$ spaces}\label{sec:intCAT}

Let $(T,\mathsf{T},\tau)$ be a probability space and $(\mathcal{H},d)$ be a separable Hadamard space. An ($\mathcal{H}$-valued) random variable is a map $x:T\to \mathcal{H}$ which is $\mathsf{T}$/$\mathsf{B}(\mathcal{H})$-measurable, where $\mathsf{B}(\mathcal{H})$ is the Borel $\sigma$-algebra of $\mathcal{H}$. Write $\mathcal{P}(\mathcal{H})$ for the set of all probability measures on $\mathcal{H}$ and for $p\in [1,\infty)$, write $\mathcal{P}^p(\mathcal{H})$ for the set of all measures $P\in\mathcal{P}(\mathcal{H})$ such that $\int d^p(w,z)\,dP(w)<\infty$ for some/any $z\in \mathcal{H}$. The push-forward measure $\tau_{x}(A):= \tau(x^{-1}(A))$ for $A\in \mathsf{B}(\mathcal{H})$ is called the distribution of $x$. Naturally, we have $\tau_x\in\mathcal{P}(\mathcal{H})$.

Fundamentally (compare e.g.\ Theorem 2.3.1 in \cite{Bacak2014a}), we have the following result in Hadamard spaces: For any $P\in\mathcal{P}^1(\mathcal{H})$ and some/any $y\in \mathcal{H}$, there is a unique minimizer 
\[
b(P):=\mathrm{argmin}_{z\in \mathcal{H}}\int\left(d^2(z,w)-d^2(w,y)\right)\,dP(w),
\]
which is independent of $y$, called the barycenter of $P$. Given $p\in [1,\infty)$, the space $\mathcal{L}^p(T,\mathcal{H},\tau)$ is the space of all $\mathsf{T}$/$\mathsf{B}(\mathcal{H})$-measurable maps $x:T\to \mathcal{H}$ with $d_p(x,z)<\infty$ for some/any $z\in \mathcal{H}$, where $d^p_p(x,y):= \int d^p(x,y)\,d\tau$ for $\mathsf{T}$/$\mathsf{B}(\mathcal{H})$-measurable maps $x,y:T\to \mathcal{H}$. The space $L^p(T,\mathcal{H},\tau)$ arises from this space by considering equivalence classes under the equivalence relation defined by $d_p(x,y)=0$. Note that $x\in L^p(T,\mathcal{H},\tau)$ if, and only if, $\tau_x\in\mathcal{P}^p(\mathcal{H})$. For $p=\infty$, the space $\mathsf{L}^\infty(T,\mathcal{H},\tau)$ is the space of all almost sure equivalence classes of $\mathsf{T}$/$\mathsf{B}(\mathcal{H})$-measurable maps $x:T\to \mathcal{H}$ where $d(x,z)\in L^\infty(T,\mathbb{R},\tau)$ for some/any $z\in \mathcal{H}$. The expectation of a random variable $x\in L^1(T,\mathcal{H},\tau)$ is now defined via 
\[
\int x\,d\tau:=b(\tau_x)
\]
where $b(\tau_x)$ is the barycenter of $\tau_x$ as before. Later, we will sometimes work with multiple probability spaces simultaneously. In that context, we generally use $(\Omega,\mathsf{F},\PP)$ as notation for a ground space over which results take place, where we use $\EE[x]$ as an alternative notation for the integral $\int x\,d\PP$, and $(E,\mathsf{E},\mu)$ as notation for a secondary space. We use the notation $(T,\mathsf{T},\tau)$ in general sections like the present to make it easier to transfer results to either case.

Before we move on, we briefly consider the above notion in the special case of a discrete measure space: If $T=\{t_1,\dots,t_n\}$ is a finite set with full $\sigma$-algebra $\mathsf{T}=2^T$, a measurable random variable $x:T\to \mathcal{H}$ can simply be identified with a tuple of points $(x_1,\dots,x_n)$ from $\mathcal{H}$. If $\tau$ is a measure on $(T,\mathsf{T})$ with $\tau(\{t_i\})=w_i$ such that $\sum_{i=1}^n w_i=1$, then the integral $\int x\,d\tau$ of $x=(x_1,\dots,x_n)$ reduces to
\[
\mathrm{argmin}_{z\in \mathcal{H}}\int d^2(z,x(t))\,d\tau(t)=\mathrm{argmin}_{z\in \mathcal{H}} \sum_{i=1}^n w_i d^2(z,x_i),
\]
also called the Fr\'echet mean of the points $(x_1,\dots,x_n)$ (see e.g.\ \cite{Bacak2014a}), a generalization of the arithmetic mean in Hadamard spaces which is of central importance for various applications. For this, we further refer to the discussions in \cite{Bacak2014a} as well as \cite{GoodwinLewisLopezAcedoNicolae2025}, the latter of which gives various computational approaches to recognizing whether a point arises as a mean itself. 

The first key theoretical result on the expectation that we rely on here is as follows: 

\begin{lemma}[Corollary 2.4 in \cite{Sturm2002}]\label{ExpNonexp}
For any $x,y\in L^1(T,\mathcal{H},\tau)$:
\[
d\left(\int x\,d\tau,\int y\,d\tau\right)\leq \int d(x,y)\,d\tau.
\]
\end{lemma}

The above lemma essentially states (and can be derived from) the nonexpansivity of the barycenter relative to the Wasserstein (or Kantorovich-Rubinstein) metric
\[
W_1(P,P'):=\inf_{\gamma\in \Gamma(P,P')}\int d(x,y)\,d\gamma(x,y)
\]
on $\mathcal{P}^1(\mathcal{H})$, where $\Gamma(P,P')$ is the set of all couplings
\[
\Gamma(P,P'):=\{\gamma\in\mathcal{P}(\mathcal{H}^2)\mid \gamma(A\times \mathcal{H})=P(A)\text{ and }\gamma(\mathcal{H}\times A)=P'(A)\text{ for all }A\in\mathsf{B}(\mathcal{H})\}.
\]

The second key result on the expectation that we rely on here is the following analogous version of Jensen's inequality:

\begin{lemma}[Proposition 3.4 in \cite{Sturm2002}]\label{nonlinJensen}
Let $x\in L^1(T,\mathcal{H},\tau)$ and let $\varphi:\mathcal{H}\to\mathbb{R}$ be lower semicontinuous and convex with $(\varphi\circ x)_+\in L^1(T,\mathbb{R},\tau)$. Then
\[
\int \varphi\circ x\,d\tau\geq \varphi\left(\int x\,d\tau\right).
\]
\end{lemma}

We refer to the fundamental works of Sturm \cite{Sturm2002,Sturm2003} as well as the exposition of these matters given by Ba\v{c}\'ak \cite{Bacak2014a} for further information on the above.

We will later use the extension of the above integral to set-valued operators. Following \cite{Pischke2025}, we define this operation in an analogous way to the seminal work of Aumann \cite{Aumann1965} set in Hilbert spaces, replacing the use of the Bochner integral therein with the integral of Sturm. Concretely, given a set-valued operator $F:T\to 2^\mathcal{H}$, a function $\phi:T\to \mathcal{H}$ is a measurable selection of $F$ if it is $\mathsf{T}$/$\mathsf{B}(\mathcal{H})$-measurable and $\phi(t)\in F(t)$ for all $t\in T$. The set of all measurable selections of $F$ is denoted by $\mathcal{S}(F)$ and we write $\mathcal{S}^p(F):=\mathcal{S}(F)\cap L^p(T,\mathcal{H},\tau)$. The Aumann-Sturm integral of $F$ is then defined as
\[
\int F\,d\tau=\left\{\int \phi\,d\tau\mid \phi\in\mathcal{S}^1(F)\right\}.
\]

Again, we briefly consider the above notion over discrete measure spaces. If $T=\{t_1,\dots,t_n\}$ is a finite set with full $\sigma$-algebra $\mathsf{T}=2^T$, a set-valued map $F:T\to 2^\mathcal{H}$ can be identified with a tuple $(F_1,\dots,F_n)$ of sets $F_i\subseteq \mathcal{H}$. If now again $\tau$ is a measure on $(T,\mathsf{T})$ with $\tau(\{t_i\})=w_i$ such that $\sum_{i=1}^n w_i=1$, then the integral $\int F\,d\tau$ reduces to
\[
\left\{\mathrm{argmin}_{z\in \mathcal{H}} \sum_{i=1}^n w_i d^2(z,x_i)\mid x_i\in F_i \text{ for all }i=1,\dots,n\right\},
\]
that is the set of all Fr\'echet means which can be formed for tuples from $F_1\times\dots\times F_n$.

As we will later be concerned with probabilistic notions on $T_x\mathcal{H}$, we will need to ensure that the respective concepts are well-defined thereon. In particular, this entails a separability assumption on $T_x\mathcal{H}$ which we briefly discuss in the following remark (compare Remark 2.18 in \cite{Pischke2025}).

\begin{remark}\label{rem:tangentSepAndMeas}
We later crucially rely on the tangent cone $T_x\mathcal{H}$ of $\mathcal{H}$ being separable. This however directly follows if the underlying Hadamard space $\mathcal{H}$ separable as discussed in Remark 2.18 in \cite{Pischke2025}. This in particular ensures that $d_x$ is jointly measurable and in particular also guarantees that $g_x$ and $\norm{\cdot}_x$ are (jointly) measurable as well. Further, note that $\log_x$ is measurable as it is nonexpansive and hence uniformly continuous.
\end{remark}

It is in particular in the context of the interaction between the integral and the pseudo-Riemannian metric that we in this paper make essential use of the assumption that $T_x\mathcal{H}$ is flat later on, that is that $\mathcal{H}$ is a Hilbert-Hadamard space. Concretely, in a separable Hilbert space $(\mathcal{H},\langle\cdot,\cdot\rangle)$, it is well-known that independent ($\mathcal{H}$-valued) random variables $u,v:T\to \mathcal{H}$ satisfy 
\[
\int \langle u,v\rangle\,d\tau=\left\langle \int u\,d\tau,\int v\,d\tau\right\rangle,
\]
where the integral in that case reduces to the Bochner integral on $\mathcal{H}$. Flatness of $T_x\mathcal{H}$ allows us to extend this property to the pseudo-Riemannian metric $g_x$, where we crucially rely on Lemma \ref{affine} above. Like that lemma, also this observation was already made and crucially employed in \cite{Pischke2025} (in fact even an extension to conditional expectations, which we will not need here).

\begin{lemma}[Lemma 2.19 in \cite{Pischke2025}]\label{intEx}
Fix $x\in \mathcal{H}$ and let $T_x\mathcal{H}$ be flat. For independent integrable random variables $u,v:T\to T_x\mathcal{H}$, we have
\[
\int g_x(u,v)\,d\tau= g_x\left(\int u\,d\tau,\int v\,d\tau\right).
\]
\end{lemma}

Lastly, we discuss the space $L^2(T,\mathcal{H},\tau)$ in more detail, which features crucially in a later section on the interchange of the subdifferential and the integral. Since at least the work of Jost (see Corollary 4.1.1 in  \cite{Jost1997} and see also e.g.\ Proposition 3.3 \cite{Sturm2001}), it is known that the space $L^2(T,\mathcal{H},\tau)$ is itself a Hadamard space whenever $\mathcal{H}$ is one. As mentioned before, it was shown by Gong, Wu and Yu (see Proposition 3.13 in \cite{GongWuYu2021}) that if $\mathcal{H}$ is a Hilbert-Hadamard space, then also $L^2(T,\mathcal{H},\tau)$ is a Hilbert-Hadamard space, and if $\mathcal{H}$ and $T$ are further separable,\footnote{A probability space $(T,\mathsf{T},\tau)$ is called separable if the metric space $(\mathsf{T},d_\tau)$ with $d_\tau(A,B)=\tau(A\triangle B)$ is separable, where $A\triangle B$ is the symmetric difference of $A,B\in\mathsf{T}$.} then so is $L^2(T,\mathcal{H},\tau)$. A study of the tangent cones of $L^2(T,\mathcal{H},\tau)$ can be found in Proposition 4.5 of \cite{GigliNobili2021}, already for much more general spaces $\mathcal{H}$, by which $T_xL^2(T,\mathcal{H},\tau)$ for $x\in L^2(T,\mathcal{H},\tau)$ is isometric to the space of $L^2$-sections of the so-called pullback $x^*T_G\mathcal{H}$ of the geometric tangent bundle $T_G\mathcal{H}$ (see \cite{GigliNobili2021} for precise definitions). In the case that $\mathcal{H}$ is a separable Hadamard space and that $(T,\mathsf{T},\tau)$ is separable (as later assumed), Proposition 4.8 of \cite{GigliNobili2021} shows that this characterization of $T_xL^2(T,\mathcal{H},\tau)$ reduces to the space of all equivalence classes of measurable $u$ such that $u(t)\in T_{x(t)}\mathcal{H}$ almost surely and $\int \norm{u(t)}^2_{x(t)}\,d\tau(t)<+\infty$. In particular, if $x\in \mathcal{H}$, then $T_xL^2(T,\mathcal{H},\tau)$ is isometric to $L^2(T,T_x\mathcal{H},\tau)$. In that case, one also in particular obtains that for $u,v\in T_xL^2(T,\mathcal{H},\tau)$ that
\[
d^2_x(u,v)=\int d^2_x(u(t),v(t))\,d\tau(t)
\]
and so $\norm{u}_x^2=\int \norm{u}_x^2\,d\tau$ as well as $g_x(u,v)=\int g_x(u(t),v(t))\,d\tau(t)$.

\section{An $L^1$ strong law of large numbers}\label{sec:SLLN}

The main goal of this section is to prove the quantitative generalization of the strong law of large numbers for $L^1$ i.i.d.\ processes where crucially, as already motivated in the introduction, the notion of sample average is that of the inductive mean, following Sturm \cite{Sturm2002,Sturm2003}:

\begin{definition}[Inductive mean]\label{IndMean}
Let $(x_i)\subseteq \mathcal{H}$ be given. The sequence of inductive means $(s_n)\subseteq \mathcal{H}$ of $(x_i)$ is defined as follows:
\[
s_1:=x_1\text{ and }s_{n}:=\left(1-\frac{1}{n}\right) s_{n-1}\oplus \frac{1}{n}x_n.
\]
We write $\frac{1}{n}\dsum_{i=1,\dots,n}x_i$ for $s_n$. In general, this point will depend on the permutation of the $(x_i)$.
\end{definition}

To derive the distribution-dependent non-asymptotic concentration inequality, we later apply a type of geodesic truncation to the random variables which allows us to leverage the Busemann convexity of the space as well as the nonexpansiveness of barycenters w.r.t.\ the Wasserstein metric to reduce the result to a previous non-asymptotic concentration inequality for the usual strong law of large numbers on the reals due to Kachurovskii \cite{Kachurovskii1996} (see Theorem \ref{L1LawRealsQuant} later on, and also the recent \cite{RufWaudbySmith2025} for related results) together with a respective quantitative variant of Sturm's strong law of large numbers for bounded i.i.d.\ sequences. We now begin with this latter ingredient.

\begin{lemma}\label{LinftyLawQuant}
Let $(\Omega,\mathsf{F},\PP)$ be a probability space and let $(\mathcal{H},d)$ be a separable Hadamard space. If $(X_i)$ is an i.i.d.\ sequence of random variables $X_i\in L^\infty(\Omega,\mathcal{H},\PP)$ with $d(o,X_1)\leq b$ $\PP$-a.s.\ for some $b\geq 1$ and $o\in \mathcal{H}$, then for any $\varepsilon>0$ and $m\in\mathbb{N}$:
\[
\PP\left(\sup_{n\geq m} d\left(\frac{1}{n}\underset{i=1,\dots,n}{\dsum} X_i,\EE[X_1]\right)>\varepsilon\right)\leq \frac{8b^2}{\min\{\varepsilon^2,1\}m^{1/2}}.
\]
\end{lemma}
\begin{proof}
Write $S_n$ for $\frac{1}{n}\underset{i=1,\dots,n}{\dsum} X_i$. We first show that for any $\varepsilon,\lambda>0$:
\[
\PP\left(\sup_{n\geq (\max\{4b^2/\varepsilon^2\lambda,8b/\varepsilon\})^2} d\left(S_n,\EE[X_1]\right)>\varepsilon\right)\leq \lambda.\tag{$*$}\label{sturmFirstQuant}
\]
As shown in the proof of Theorem 4.7 in \cite{Sturm2003}, it holds that 
\[
\sum_{n=1}^\infty\PP(d(S_{n^2},\EE[X_1])>\varepsilon)\leq \sum_{n=1}^\infty \frac{\mathbb{V}[X_1]}{\varepsilon^2n^2},
\]
where $\mathbb{V}[X_1]:=\inf_{z\in \mathcal{H}}\EE[d^2(z,X_1)]$ is the variance of $X_1$. Note that $\mathbb{V}[X_1]\leq \EE[d^2(o,X_1)]\leq b^2$. As the measure is sub-$\sigma$-additive, we get
\begin{align*}
\PP(\sup_{n\geq n_0}d(S_{n^2},\EE[X_1])>\varepsilon)&\leq \sum_{n=\lceil n_0\rceil}^\infty\PP(d(S_{n^2},\EE[X_1])>\varepsilon)\\
&\leq \frac{b^2}{\varepsilon^2}\sum_{n=\lceil n_0\rceil}^\infty\frac{1}{n^2}\leq \frac{2b^2}{\varepsilon^2n_0}.
\end{align*}
As further shown in the proof of Theorem 4.7 in \cite{Sturm2003}, one has $d(S_{n^2},S_k)\leq 4b/n$ for all $n,k\geq 1$ with $n^2\leq k <(n+1)^2$. For $\varepsilon,\lambda>0$ now fixed, write
\[
n_0:=\max\left\{\frac{8b^2}{\varepsilon^2\lambda},\frac{8b}{\varepsilon}\right\}.
\]
For $\omega\in\Omega$ such that $d(S_{n^2}(\omega),\EE[X_1])\leq\varepsilon/2$ for all $n\geq n_0$ and $d(S_{n^2}(\omega),S_k(\omega))\leq 4b/n$ for all $n,k\geq 1$ as above, if $k\geq n_0^2$, then $k\in [n^2,(n+1)^2)$ for some $n\geq n_0$. Hence, for such a $k$, we have
\[
d(S_{k}(\omega),\EE[X_1])\leq d(S_k(\omega),S_{n^2}(\omega))+d(S_{n^2}(\omega),\EE[X_1])\leq \frac{4b}{n}+\frac{\varepsilon}{2}\leq \frac{4b}{n_0}+\frac{\varepsilon}{2}\leq \varepsilon.
\]
Therefore, we get
\[
\PP\left(\sup_{k\geq n_0^2} d\left(S_k,\EE[X_1]\right)>\varepsilon\right)\leq \PP\left(\sup_{n\geq n_0} d\left(S_{n^2},\EE[X_1]\right)>\varepsilon/2\right)\leq \frac{8b^2}{\varepsilon^2n_0}\leq \lambda.
\]
This shows \eqref{sturmFirstQuant}. To now see the inequality claimed in the theorem, note first that \eqref{sturmFirstQuant} implies
\[
\PP\left(\sup_{n\geq n_\lambda^2} d\left(S_n,\EE[X_1]\right)>\varepsilon\right)\leq \lambda
\]
for $n_\lambda:=8b^2/(\min\{\varepsilon^2,1\}\min\{\lambda,1\})$. Now, if $m^{1/2}\geq 8b^2/\min\{\varepsilon^2,1\}$, we obtain
\[
\PP\left(\sup_{n\geq m} d\left(S_n,\EE[X_1]\right)>\varepsilon\right)\leq \frac{8b^2}{\min\{\varepsilon^2,1\}m^{1/2}}
\] 
by setting $\lambda:=8b^2/(\min\{\varepsilon^2,1\}m^{1/2})$ in the above. However, as $8b^2/(\min\{\varepsilon^2,1\}m^{1/2})> 1$ if $m^{1/2}< 8b^2/\min\{\varepsilon^2,1\}$, the above inequality holds unconditionally.
\end{proof}

As mentioned before, we similarly require a quantitative estimate for the usual $L^1$ strong law of large numbers on $\mathbb{R}$, which takes the form of the following distribution-dependent non-asymptotic concentration inequality established by Kachurovskii \cite{Kachurovskii1996}:

\begin{theorem}[essentially Theorem 26 in \cite{Kachurovskii1996}]\label{L1LawRealsQuant}
Let $(\Omega,\mathsf{F},\PP)$ be a probability space and let $(X_i)$ be an i.i.d.\ sequence of random variables $X_i\in L^1(\Omega,\mathbb{R},\PP)$. Then, for any $\varepsilon>0$ and $m\in\mathbb{N}$:
\[
\PP\left(\sup_{n\geq m}\left\vert\frac{1}{n}\sum_{i=1}^n X_i- \EE[X_1]\right\vert>\varepsilon\right)\leq \frac{2}{\varepsilon}\left(\frac{\sqrt{\EE[\vert X_1\vert]}}{m^{1/4}}+U_\PP(m^{1/2})\right),
\]
where $U_\PP(x):=\EE[\vert X_1\vert\mathbf{1}_{\vert X_1\vert> x}]$ is the truncated first absolute moment of $X_1$.
\end{theorem}

Note that the right-hand side of the inequality from Theorem \ref{L1LawRealsQuant} vanishes as $m\to\infty$ since $U_\PP(x)\to 0$ as $x\to\infty$ by the dominated convergence theorem.

Combined, we can now derive the following quantitative estimate on the asymptotic behavior of $L^1$ i.i.d.\ processes on $\mathcal{H}$:

\begin{theorem}\label{L1LawHadQuant}
Let $(\Omega,\mathsf{F},\PP)$ be a probability space and let $(\mathcal{H},d)$ be a separable Hadamard space. If $(X_i)$ is an i.i.d.\ sequence of random variables $X_i\in L^1(\Omega,\mathcal{H},\PP)$, then for any $\varepsilon>0$ and any $m\in\mathbb{N}$:
\[
\PP\left(\sup_{n\geq m} d\left(\frac{1}{n}\underset{i=1,\dots,n}{\dsum} X_i,\EE[X_1]\right)>\varepsilon\right)\leq \frac{6}{\varepsilon}\left(\frac{\sqrt{\EE[d(o,X_1)]}}{m^{1/4}}+U_\PP(m^{1/8})\right)+\frac{72}{\min\{\varepsilon^2,1\}m^{1/4}},
\]
where $U_\PP(x):=\EE[d(o,X_1)\mathbf{1}_{d(o,X_1)\geq x}]$ for a fixed $o\in \mathcal{H}$.
\end{theorem}

Also here, the right-hand side of the inequality featuring in Theorem \ref{L1LawHadQuant} vanishes as $m\to\infty$.

\begin{proof}[Proof of Theorem \ref{L1LawHadQuant}]
Let $b\geq 1$ be given. We associate with $(X_i)$ a bounded (relative to $b$ and $o$) and i.i.d.\ process $(X_i^b)$ by truncating it along geodesics as follows:
\[
X_i^b:=\begin{cases}X_i&\text{if }d(o,X_i)\leq b,\\
\frac{b}{d(o,X_i)}X_i\oplus \frac{d(o,X_i)-b}{d(o,X_i)}o&\text{if }d(o,X_i)>b.\end{cases}
\]
For brevity, write 
\[
S_n:=\frac{1}{n}\underset{i=1,\dots,n}{\dsum} X_i\text{ and }S^b_n:=\frac{1}{n}\underset{i=1,\dots,n}{\dsum} X_i^b.
\]
Note that $d(o,X^b_i)=\min\{d(o,X_i),b\}$ so that $(X_i^b)$ is i.i.d.\ as well as $X_i^b\in L^\infty(\Omega,\mathcal{H},\PP)$. Our quantitative version of Sturm's $L^\infty$ strong law of large numbers, that is Lemma \ref{LinftyLawQuant}, then yields
\[
\PP\left(\sup_{n\geq m} d\left(S_n^b,\EE[X^b_1]\right)>\varepsilon\right)\leq \frac{8b^2}{\min\{\varepsilon^2,1\}m^{1/2}}
\]
for any $\varepsilon>0$ and $m\in\mathbb{N}$. Now note that since $\mathcal{H}$ is a Hadamard space, it is a Busemann space and so we have that $d$ is jointly convex. In particular, we inductively get
\[
d\left(\frac{1}{n}\underset{i=1,\dots,n}{\dsum} x_i,\frac{1}{n}\underset{i=1,\dots,n}{\dsum} y_i\right)\leq\frac{1}{n}\sum_{i=1}^n d(x_i,y_i)
\]
for any pair of sequences $(x_i),(y_i)\subseteq \mathcal{H}$. This yields
\[
d\left(S_n,S_n^b\right)\leq\frac{1}{n}\sum_{i=1}^n d(X_i,X_i^b).
\]
Since we have $d(X_i,X_i^b)=(d(o,X_i)-b)^+$, the real-valued process $(d(X_i,X_i^b))$ is i.i.d.\ and we have $d(X_i,X_i^b)\in L^1(\Omega,\mathbb{R},\PP)$ since $X_i\in L^1(\Omega,\mathcal{H},\PP)$. Appealing to the above quantitative variant of the real-valued $L^1$ strong law of large numbers, that is Theorem \ref{L1LawRealsQuant}, now yields
\[
\PP\left(\sup_{n\geq m}\left\vert\frac{1}{n}\sum_{i=1}^n d(X_i,X_i^b)- \EE[(d(o,X_1)-b)^+]\right\vert>\varepsilon\right)\leq \frac{2}{\varepsilon}\left(\frac{\sqrt{\EE[d(o,X_1)]}}{m^{1/4}}+V^b_\PP(m^{1/2})\right)
\]
for any $\varepsilon>0$ and $m\in\mathbb{N}$, where $V^b_\PP(x):=\EE[(d(o,X_1)-b)^+\mathbf{1}_{(d(o,X_1)-b)^+> x}]$. Using the above inequality, we hence get
\begin{align*}
d\left(S_n,S_n^b\right)&= d\left(S_n,S_n^b\right)-\EE[(d(o,X_1)-b)^+]+\EE[(d(o,X_1)-b)^+]\\
&\leq \left\vert \frac{1}{n}\sum_{i=1}^n d(X_i,X_i^b)-\EE[(d(o,X_1)-b)^+]\right\vert+\EE[(d(o,X_1)-b)^+].
\end{align*}
Lemma \ref{ExpNonexp} now yields that
\[
d(\EE[X_1],\EE[X_1^b])\leq \EE[d(X_1,X_1^b)]=\EE[(d(o,X_1)-b)^+].
\]
The triangle inequality yields
\begin{align*}
d\left(S_n,\EE[X_1]\right)&\leq d\left(S_n,S_n^b\right) + d\left(S_n^b,\EE[X_1^b]\right)+d(\EE[X_1^b],\EE[X_1])\\
&\leq \left\vert \frac{1}{n}\sum_{i=1}^n d(X_i,X_i^b)-\EE[(d(o,X_1)-b)^+]\right\vert+d\left(S_n^b,\EE[X_1^b]\right)+2\EE[(d(o,X_1)-b)^+].
\end{align*}
Now, note that $(d(o,X_1)-b)^+=(d(o,X_1)-b)\mathbf{1}_{d(o,X_1)\geq b}$ and so
\[
\EE[(d(o,X_1)-b)^+]=U_\PP(b)-b\PP(d(o,X_1)\geq b)\leq U_\PP(b).
\]
In particular, we have $V^b_\PP(x)\leq U_\PP(b)$. Hence $2\EE[(d(o,X_1)-b)^+]\leq 2U_\PP(b)$. If $U_\PP(b)\leq \varepsilon/6$ for $b=m^{1/8}$, then we have
\begin{align*}
\PP\left(\sup_{n\geq m} d\left(S_n,\EE[X_1]\right)>\varepsilon\right)&\leq \PP\left(\sup_{n\geq m}\left\vert\frac{1}{n}\sum_{i=1}^n d(X_i,X_i^b)- \EE[(d(o,X_1)-b)^+]\right\vert>\varepsilon/3\right)\\
&\hphantom{\leq\,}+\PP\left(\sup_{n\geq m} d\left(S_n^b,\EE[X^b_1]\right)>\varepsilon/3\right)\\
&\leq\frac{6}{\varepsilon}\left(\frac{\sqrt{\EE[d(o,X_1)]}}{m^{1/4}}+V^b_\PP(m^{1/2})\right)+\frac{72b^2}{\min\{\varepsilon^2,1\}m^{1/2}}\\
&\leq \frac{6}{\varepsilon}\left(\frac{\sqrt{\EE[d(o,X_1)]}}{m^{1/4}}+U_\PP(m^{1/8})\right)+\frac{72}{\min\{\varepsilon^2,1\}m^{1/4}}.
\end{align*}
If $\varepsilon/6< U_\PP(m^{1/8})$, then $\frac{6}{\varepsilon}U_\PP(m^{1/8})>1$ so that the result holds unconditionally.
\end{proof}

In particular, the above Theorem \ref{L1LawHadQuant} entails the purely qualitative strong law of large numbers for $L^1$ i.i.d.\ processes first established by Yokota \cite{Yokota2018}:

\begin{corollary}\label{L1LawHad}
Let $(\Omega,\mathsf{F},\PP)$ be a probability space and let $(\mathcal{H},d)$ be a separable Hadamard space. If $(X_i)$ is an i.i.d.\ sequence of random variables $X_i\in L^1(\Omega,\mathcal{H},\PP)$, then
\[
\frac{1}{n}\underset{i=1,\dots,n}{\dsum} X_i\to \EE[X_1]\quad \PP\text{-a.s.}
\]
as $n\to\infty$.
\end{corollary}

However, akin to \cite{RufWaudbySmith2025}, also here the concentration inequality established in Theorem \ref{L1LawHadQuant} is strong enough to immediately entail a distribution-uniform generalization of the strong law of large numbers in the style of Chung \cite{Chung1951} on Hadamard spaces:

\begin{corollary}\label{chungHad}
Let $(\Omega,\mathsf{F})$ be a measurable space and let $(\mathcal{H},d)$ be a separable Hadamard space. Let $\mathcal{P}$ be a collection of probability measures on $(\Omega,\mathsf{F})$ and $(X_i)$ be a sequence of random variables which is i.i.d.\ for any $\PP\in\mathcal{P}$. If $X_1$ is $\mathcal{P}$-uniformly integrable, i.e.
\[
\lim_{x\to\infty}\sup_{\PP\in\mathcal{P}}\EE_{\PP}[d(o,X_1)\mathbf{1}_{d(o,X_1)\geq x}]=0
\]
for a fixed $o\in \mathcal{H}$, then 
\[
\lim_{m\to\infty}\sup_{\PP\in\mathcal{P}}\PP\left(\sup_{n\geq m} d\left(\frac{1}{n}\underset{i=1,\dots,n}{\dsum} X_i,\EE_\PP[X_1]\right)>\varepsilon\right)=0.
\]
\end{corollary}
\begin{proof}
The corollary follows immediately from Theorem \ref{L1LawHadQuant} once one notes that the $\mathcal{P}$-uniform integrability implies that $\sup_{\PP\in\mathcal{P}}\EE_{\PP}[d(o,X_1)]<+\infty$. To see this, use the $\mathcal{P}$-uniform integrability to pick an $x$ such that $\sup_{\PP\in\mathcal{P}}\EE_{\PP}[d(o,X_1)\mathbf{1}_{d(o,X_1)\geq x}]\leq 1$. Then, for any $\PP\in\mathcal{P}$:
\[
\EE_{\PP}[d(o,X_1)]\leq \EE_{\PP}[d(o,X_1)\mathbf{1}_{d(o,X_1)\leq x}]+\EE_{\PP}[d(o,X_1)\mathbf{1}_{d(o,X_1)\geq x}]\leq x+1
\]
which yields $\sup_{\PP\in\mathcal{P}}\EE_{\PP}[d(o,X_1)]\leq x+1<+\infty$.
\end{proof}

The convergence result of Corollary \ref{L1LawHad} does not extend to the stochastic proximal point algorithm for strongly convex functions, one of the generalizations of the strong law of large numbers: While it is well-known that the stochastic proximal point algorithm for such functions on Hadamard spaces converges under $L^\infty$- or $L^2$-Lipschitz assumptions (see \cite{OhtaPalfia2015} or \cite{PischkePowell2026}, respectively), which in particular is used in \cite{OhtaPalfia2015} to re-derive Sturm's $L^\infty$ strong law of large numbers, the following example shows that this does not extend to an $L^1$-Lipschitz assumption on the function.

\begin{example}\label{exPPA}
We construct a probability space $(E,\mathsf{E},\mu)$ and a proper Hadamard space $(\mathcal{H},d)$ together with a function $f:E\times \mathcal{H}\to\mathbb{R}$ such that $f$ is $\mathsf{E}\otimes\mathsf{B}(\mathcal{H})$-measurable and $f(e,\cdot)$ is strongly convex and $L(e)$-Lipschitz with $L\in L^1(E,(0,\infty),\mu)$, but where the stochastic proximal point iteration
\[
x_{n+1}:=\mathrm{prox}_{\lambda_n f}(\xi_{n},x_n):=\mathrm{argmin}_{y\in \mathcal{H}}\left\{f(\xi_n,y)+\frac{1}{2\lambda_n}d^2(x_n,y)\right\}
\]
does not converge almost surely, for a given sequence of parameters $\lambda_n\in (0,\infty)$ such that $\sum_{n\in\mathbb{N}}\lambda_n=+\infty$ and $\sum_{n\in\mathbb{N}}\lambda_n^2<+\infty$ as well as an i.i.d.\ sequence $(\xi_n)$ of random variables $\xi_n:\Omega\to E$ with distribution $\mu$.

For that, consider the Borel space over $E:=\mathbb{R}$ with a measure $\mu$ induced by the push-forward of a symmetric and integrable random variable $\xi:\Omega\to E$ with symmetric Pareto tails (see e.g.\ \cite{Arnold2015}), that is $\EE[\xi]=0$, $\EE[\vert\xi\vert]<+\infty$ as well as
\[
\PP(\xi>t)=\frac{1}{2}t^{-r}\text{ and }\PP(\xi<-t)=\frac{1}{2}t^{-r}
\]
for all $t\geq 1$, where $r>1$. On that space and for $\mathcal{H}:=[-1,1]$, consider $f:E\times \mathcal{H}\to\mathbb{R}$ defined by $f(e,x):=\frac{1}{2}x^2-ex$. Then $f(e,\cdot)$ is strongly convex and
\[
\vert f(e,x)-f(e,y)\vert\leq (1+\vert e\vert)\vert x-y\vert
\] 
for all $x,y\in [-1,1]$ and $e\in\mathbb{R}$, so that $f(e,\cdot)$ is Lipschitz with constant $L(e):=1+\vert e\vert\in L^1(E,(0,\infty),\mu)$. Further, $\underline{f}(x):=\int f(e,x)\,d\mu(e)=\frac{1}{2}x^2$, which has the unique minimizer $0$ on $[-1,1]$. It is easy to see that
\[
\mathrm{prox}_{\lambda f(e,\cdot)}(x)=P_{[-1,1]}\left(\frac{x+\lambda e}{1+\lambda}\right),
\]
where $P_{[-1,1]}$ is the projection onto the interval $[-1,1]$. Now, for $\lambda_n=1/n^a$ with $a\in (\frac{1}{2},1)$ and where $r\in (1,\frac{1}{a}]$, we have $\sum_{n\in\mathbb{N}}\lambda_n=+\infty$ and $\sum_{n\in\mathbb{N}}\lambda_n^2<+\infty$. Let now $(\xi_n)$ be i.i.d.\ with distribution $\mu$. We then have
\[
\sum_{n\in\mathbb{N}}\PP\left(\xi_{n}>\frac{3}{\lambda_n}\right)=\sum_{n\in\mathbb{N}}\frac{1}{2\cdot 3^r}\lambda_n^r=+\infty,
\]
and so the Borel-Cantelli lemma yields that $\xi_{n}>\frac{3}{\lambda_n}$ holds infinitely often almost surely. Now if $\xi_{n}>\frac{3}{\lambda_n}$, we have $\frac{x+\lambda e}{1+\lambda}>1$ and so $x_{n+1}=1$, which hence holds infinitely often almost surely. Similarly, we get $\sum_{n\in\mathbb{N}}\PP\left(\xi_{n}<-\frac{3}{\lambda_n}\right)=+\infty$ and so $x_{n+1}=-1$ holds infinitely often almost surely. Hence $x_{n+1}$ cannot converge almost surely.
\end{example}

As already highlighted in the introduction, it still remains open whether an $L^1$ version of the full Birkhoff pointwise ergodic theorem holds for inductive means. Concretely, while various ergodic theorems are known on Hadamard spaces (see e.g.\ \cite{AntezanaGhiglioniStojanoff2023,Austin2011,ChoiKim2022,Navas2013}, as well as \cite{KostenbergerStark2025}), similar as with the strong law previously, all of these results deviate in some way from a full formulation of Birkhoff's pointwise ergodic theorem on Hadamard spaces using inductive means, with \cite{AntezanaGhiglioniStojanoff2023,ChoiKim2022}, while using inductive means, not being phrased for general dynamical systems but only Kronecker systems (which in particular exclude the Bernoulli shift otherwise used to derive the strong law of large numbers from the ergodic theorem) and with \cite{Austin2011,Navas2013} using other notions of averages based on (canonical) barycenters. We thus want to highlight the following question:

\begin{question}\label{ergodicQuestion}
Let $(\Omega,\mathsf{F},\PP)$ be a probability space and let $(\mathcal{H},d)$ be a separable Hadamard space. Suppose $T:\Omega\to\Omega$ is a measurable and measure-preserving function and that $f\in L^1(\Omega,\mathcal{H},\PP)$. Does it then hold that the inductive ergodic averages
\[
\frac{1}{n}\underset{i=1,\dots,n}{\dsum} f(T^{(i)}(\omega)),
\]
converge to $\EE[f\mid \mathsf{I}]$ almost surely? Here, $T^{(i)}(\omega)$ denotes $i$-fold application of $T$ to $\omega$, $\mathsf{I}=\{A\in\mathsf{F}\mid T^{-1}(A)=A\}$ is the $\sigma$-algebra of invariant events and $\EE[f\mid \mathsf{I}]$ is the conditional expectation of $f$ under $\mathsf{I}$ in the sense of Sturm \cite{Sturm2002}.
\end{question}

The above is in particular known in Hilbert spaces (see e.g.\ \cite{Krengel2011}). A result such as in Question \ref{ergodicQuestion} would then in particular imply the strong law of large numbers from Corollary \ref{L1LawHad}. In particular, the strategy employed above to deduce the $L^1$ strong law of large numbers is also viable in this context, that is using a similar bounding technique along geodesics, one can derive an $L^1$ variant of the pointwise ergodic theorem already from the $L^\infty$ version of the pointwise ergodic theorem in Hadamard spaces together with the usual Birkhoff pointwise ergodic theorem on $\mathbb{R}$.

Another open problem already highlighted in the introduction is the generalization of Corollary \ref{L1LawHad} to weaker notions than i.i.d.\ random variables, such as e.g.\ pairwise i.i.d.\ random variables.

\begin{question}\label{pairwiseQuestion}
Let $(\Omega,\mathsf{F},\PP)$ be a probability space and let $(\mathcal{H},d)$ be a separable Hadamard space. If $(X_i)$ is a sequence of pairwise independent and identically distributed random variables $X_i\in L^1(\Omega,\mathcal{H},\PP)$, does it then hold that
\[
\frac{1}{n}\underset{i=1,\dots,n}{\dsum} X_i\to \EE[X_1]\quad \PP\text{-a.s.}
\]
as $n\to\infty$? 
\end{question}

This problem is in particular open for $X_i\in L^\infty(\Omega,\mathcal{H},\PP)$, which suffices to establish the claim for $X_i\in L^1(\Omega,\mathcal{H},\PP)$, as also here the strategy employed above to deduce the $L^1$ strong law of large numbers, that is using a similar bounding technique along geodesics, is viable to derive an $L^1$ variant from an $L^\infty$ version together with the $L^1$ strong law of large numbers for pairwise i.i.d.\ processes (as established by Etemadi \cite{Etemadi1981}, see also e.g.\ Theorem 5.17 in \cite{Klenke2020}). In particular, if one could establish a similar concentration inequality for an $L^\infty$ strong law for pairwise i.i.d.\ processes over Hadamard spaces like that obtained in Lemma \ref{LinftyLawQuant}, then this result can be combined in a similar way as in the proof of Theorem \ref{L1LawHadQuant} with the respective concentration inequality for the real-valued $L^1$ strong law of large numbers for pairwise i.i.d.\ processes obtained in \cite{NguyenPham2021} (see Corollary 3 therein).

Next to the previous two questions, various other surrounding questions remain interesting points for future study where the above approach could be of help, like e.g.\ whether the strong law of large numbers extends to different classes of spaces such as e.g.\ Gromov's hyperbolic spaces (see the recent \cite{Ohta2024}), or whether there is an analogous result on Hadamard spaces to the strong law of large numbers by Marcinkiewicz and Zygmund \cite{MarcinkiewiczZygmund1937}, or even whether the associated non-asymptotic concentration inequality recently established in \cite{RufWaudbySmith2025} (see Theorem 2.3 therein) can be (appropriately) extended to a more geometric case. Such a result could then perhaps be used to establish, similar to \cite{RufWaudbySmith2025}, related Baum-Katz results \cite{BaumKatz1965} (see also \cite{Neri2025a}) on Hadamard spaces.

\section{Random monotone vector fields on Hadamard spaces}\label{sec:randOp}

\subsection{Monotone vector fields on Hadamard spaces}

Following \cite{ChaipunyaKohsakaKumam2021}, a monotone vector field over a Hadamard space $(\mathcal{H},d)$ is a mapping $A:\mathcal{H}\to 2^{T\mathcal{H}}$ such that $A(x)\subseteq T_x\mathcal{H}$ and 
\[
g_x(u,\log_x y)\leq - g_y(v,\log_y x)
\] 
for all $(x,u),(y,v)\in A$. We denote the set of zeros of $A$ by $\mathrm{zer}A:=\{x\in \mathcal{H}\mid 0_x\in A(x)\}$. As already discussed in the introduction, this notion simultaneously generalizes monotone operators on Hilbert spaces and monotone vector fields on Hadamard manifolds as introduced in \cite{Nemeth1999,DaCruzNetoFerreiraLucambioPerez2000} (see also \cite{LiLopezMartinMarquez2009,LiLopezMartinMarquezWang2011,WangLopezMartinMarquezLi2010}).

The key derived object for a monotone vector field is its resolvent, which we defined (following \cite{ChaipunyaKohsakaKumam2021}, which in turn generalizes \cite{LiLopezMartinMarquezWang2011}) via
\[
J^A_\lambda x:=\{z\in \mathcal{H}\mid \tfrac{1}{\lambda}\log_zx\in A(z)\}.
\]
Lemma \ref{tangentCat0} immediately yields that a $z\in \mathcal{H}$ such that $\frac{1}{\lambda}\log_zx\in A(z)$ is unique whenever $A$ is monotone. In that case, we identify $J^A_\lambda$ with the corresponding (potentially partial) function from $\mathcal{H}$ to $\mathcal{H}$. Following \cite{ChaipunyaKohsakaKumam2021}, we say that $A$ satisfies the surjectivity condition if all resolvents are total functions. In the following, we write $\mathcal{M}(\mathcal{H})$ for the set of all monotone vector fields on $\mathcal{H}$ satisfying the surjectivity condition (adapting the notation used in \cite{Salim2023}). Further, a monotone vector field $A$ is called maximal if its graph $\mathrm{gra}A:=\{(x,u)\in \mathcal{H}\times T\mathcal{H}\mid u\in Ax\}$ cannot be extended properly while preserving monotonicity.

\begin{remark}\label{rem:maxRemark}
By the well-known theorem of Minty \cite{Minty1962}, the maximality of a monotone operator over a Hilbert space is equivalent to the totality of the resolvent. Assuming that $\Sigma'_x\mathcal{H}=\Sigma_x\mathcal{H}$ for all $x\in \mathcal{H}$ (such as is e.g.\ the case if $\mathcal{H}$ is proper and has the geodesic extension property), one direction remains valid, as shown in Proposition 3.5 in \cite{ChaipunyaKohsakaKumam2021}: under this assumption, if $A$ satisfies the surjectivity condition, then it is maximal. The converse extends to the setting of Hadamard manifolds as shown in \cite{LiLopezMartinMarquez2009} under the condition that the domain $\mathrm{dom}A:=\{x\in \mathcal{H}\mid A(x)\neq\emptyset\}$ is the whole space (see Remark 4.4 in \cite{LiLopezMartinMarquez2009}). 
\end{remark}

\begin{question}
Let $A$ be a maximally monotone vector field on a Hadamard space such that $\mathrm{dom}A=\mathcal{H}$. Does $A$ then satisfy the surjectivity condition? Further, can the assumption $\mathrm{dom}A=\mathcal{H}$ be omitted? This latter problem in particular seems to be open already for Hadamard manifolds.
\end{question}

\begin{example}[cf.\ Examples 2.7 and 2.11 in \cite{Pischke2025}]\label{subDiffEx}
Let $f:\mathcal{H}\to (-\infty,+\infty]$ be a proper, convex and lower semicontinuous (lsc) function. Then the subdifferential 
\[
\partial f(x):=\{u\in T_x\mathcal{H}\mid f(y)\geq f(x)+g_x(u,\log_x y)\text{ for all }y\in \mathcal{H}\},
\]
as defined in the general setting of Hadamard spaces in \cite{ChaipunyaKohsakaKumam2021}, is a monotone vector field with $\mathrm{argmin} f:=\{x^*\in \mathcal{H}\mid f(x^*)\leq f(x)\text{ for all }x\in \mathcal{H}\}=\mathrm{zer} \partial f$. As shown in \cite[Proposition 3.8]{ChaipunyaKohsakaKumam2021}, the resolvent of $\partial f$ is given by
\[
\mathrm{prox}_{\lambda f}(x):=\mathrm{argmin}_{y\in \mathcal{H}}\left\{ f(y)+\frac{1}{2\lambda}d^2(x,y)\right\},
\]
that is the proximal map of $f$ (also called the Moreau-Yosida resolvent). Each $\mathrm{prox}_{\lambda f}$ is total (see \cite[Lemma 2]{Jost1995}), so that $\partial f$ satisfies the surjectivity condition.
\end{example}

In the rest of this paper, we actually require very few properties of the resolvent. The first is the following lemma:

\begin{lemma}[Proposition 4.3 in \cite{ChaipunyaKohsakaKumam2021}]\label{ResProp}
If $A$ is a monotone vector field, then $J^A_\lambda$ is nonexpansive, that is
\[
d(J^A_\lambda(x),J^A_\lambda(y))\leq d(x,y)
\]
for all $x,y\in\mathrm{dom}(J^A_\lambda)$. Further, it holds that $\mathrm{Fix}(J^A_\lambda)=\mathrm{zer}A$ and if $A$ also satisfies the surjectivity condition, then
\[
J^A_\lambda(x)=J^A_\gamma\left(\left(1-\frac{\gamma}{\lambda}\right)J^A_\lambda(x)\oplus \frac{\gamma}{\lambda}x\right)
\]
for any $x\in \mathcal{H}$ and $0<\gamma\leq\lambda$.
\end{lemma}

The only other property required later is the continuity of the resolvent in its real parameter. This was established as Proposition 4.12 in \cite{ChaipunyaKohsakaKumam2021} under the condition that $\mathcal{H}$ satisfies the geodesic extension property, but we here give the following lemma which disposes of that assumption:

\begin{lemma}\label{contRes}
Let $A$ be a monotone vector field that satisfies the surjectivity condition. Then for any $x\in \mathcal{H}$ and any $0<\gamma\leq\lambda$:
\[
d(J^A_\lambda(x),J^A_\gamma(x))\leq \left( 1-\frac{\gamma}{\lambda}\right) d(J^A_\lambda(x),x).
\]
In particular, $J^A_\lambda(x)$ is continuous in $\lambda\in (0,\infty)$.
\end{lemma}
\begin{proof}
Let $x\in \mathcal{H}$ and $0<\gamma\leq\lambda$. Using Lemma \ref{ResProp}, we get
\begin{align*}
d(J^A_\lambda(x),J^A_\gamma(x))&= d\left(J^A_\gamma\left(\left(1-\frac{\gamma}{\lambda}\right)J^A_\lambda(x)\oplus \frac{\gamma}{\lambda}x\right),J^A_\gamma(x)\right)\\
&\leq d\left(\left(1-\frac{\gamma}{\lambda}\right)J^A_\lambda(x)\oplus \frac{\gamma}{\lambda}x,x\right)\\
&\leq \left( 1-\frac{\gamma}{\lambda}\right) d(J^A_\lambda(x),x).
\end{align*}
To see that $J^A_\lambda(x)$ is continuous in $\lambda$, note now first that the above yields
\begin{align*}
d(J^A_\gamma(x),x)&\leq d(J^A_\gamma(x),J^A_\lambda(x))+d(J^A_\lambda(x),x)\\
&\leq \left( 1-\frac{\gamma}{\lambda}\right) d(J^A_\lambda(x),x)+d(J^A_\lambda(x),x)\\
&= \left(2-\frac{\gamma}{\lambda}\right) d(J^A_\lambda(x),x)\leq 2d(J^A_\lambda(x),x).
\end{align*}
If now $\gamma_n\to\lambda$ from below, then the first inequality yields
\[
d(J^A_\lambda(x),J^A_{\gamma_n}(x))\leq \left( 1-\frac{{\gamma_n}}{\lambda}\right) d(J^A_\lambda(x),x)\to 0
\]
directly. If $\gamma_n\to\lambda$ from above, then w.l.o.g.\ $\gamma_n\in[\lambda,\Gamma]$ for some $\Gamma>0$. Then, the first together with second inequality imply
\[
d(J^A_\lambda(x),J^A_{\gamma_n}(x))\leq \left( 1-\frac{\lambda}{{\gamma_n}}\right) d(J^A_{\gamma_n}(x),x)\leq 2\left( 1-\frac{\lambda}{{\gamma_n}}\right) d(J^A_\Gamma(x),x)\to 0.
\]
Combined, $J^A_\lambda(x)$ is continuous in $\lambda$.
\end{proof}

\subsection{Random monotone vector fields and their means}

Fix a probability space $(T,\mathsf{T},\tau)$ and let $(\mathcal{H},d)$ be a separable Hadamard space. Over these spaces, consider a map $A:T\to\mathcal{M}(\mathcal{H})$.\footnote{Equivalently, we can regard $A$ as a perturbed set-valued vector field $A:T\times \mathcal{H}\to 2^{T\mathcal{H}}$ with $A(t,x)\subseteq T_x\mathcal{H}$ for all $t\in T$ and $x\in \mathcal{H}$ such that $A(t,\cdot)$ is monotone and satisfies the surjectivity condition for any $t\in T$ (as studied in \cite{Pischke2025}).} We write $A(t,x)$ for $A(t)(x)$.

For such a perturbed vector field, we define its resolvent similar to before via 
\[
J^{A(t)}_\lambda(x):=\{z\in \mathcal{H}\mid \tfrac{1}{\lambda}\log_zx\in A(t,z)\}
\]
for $\lambda>0$, $t\in T$ and $x\in \mathcal{H}$. The following notion of a random monotone vector field, characterized by a measurability condition on these perturbed resolvents, was recently introduced in \cite{Pischke2025}, extending the notion used in Hilbert spaces (see \cite{Bianchi2016} as well as \cite{Salim2023}, among others):

\begin{definition}
An operator $A:T\to \mathcal{M}(\mathcal{H})$ is called a random monotone vector field if $J^{A(\cdot)}_\lambda(x)$ is $\mathsf{T}$/$\mathsf{B}(\mathcal{H})$-measurable for any $x\in \mathcal{H}$ and any $\lambda>0$.
\end{definition}

As $J^{A(t)}_\lambda(\cdot)$ is single-valued, total and uniformly continuous (being nonexpansive, as discussed before) for any $t\in T$, the map $J^{A(\cdot)}_\lambda(\cdot)$ is therefore a Carath\'eodory map and, in particular, is $\mathsf{T}\otimes\mathsf{B}(\mathcal{H})$/$\mathsf{B}(\mathcal{H})$-measurable (see e.g.\ Lemma 8.2.6 in \cite{AubinFrankowska2009}). 

\begin{example}\label{ex:RMVF}
Let $f:T\times \mathcal{H}\to (-\infty,+\infty]$ be a given function. Following \cite{RockafellarWets1998}, we call $f$ a normal integrand if $f(t,\cdot)$ is lsc for all $t\in T$ and its epigraph $\mathrm{epi}f(t,\cdot):=\{(x,\alpha)\in \mathcal{H}\times\mathbb{R}\mid f(t,x)\leq\alpha\}$ is measurable as a set-valued function $T\to 2^{\mathcal{H}\times\mathbb{R}}$. We call $f$ a normal convex integrand if $f(t,\cdot)$ is additionally convex for all $t\in T$ and we call $f$ proper if $f(t,\cdot)$ is proper for any $t\in T$. If $(T,\mathsf{T},\tau)$ is complete, the measurability of the epigraph condition can equivalently be replaced with the assumption that $f$ is $\mathsf{T}\otimes\mathsf{B}(\mathcal{H})$-measurable.

For such a proper normal convex integrand $f$, consider now the associated random subdifferential 
\[
\partial f(t,x):=\{u\in T_x\mathcal{H}\mid f(t,y)\geq f(t,x)+g_x(u,\log_x y)\text{ for all }y\in \mathcal{H}\}
\]
for $t\in T$ and $x\in\mathcal{H}$. As discussed in Example \ref{subDiffEx}, the resolvents of $\partial f(t,\cdot)$ are given by the proximal maps $\mathrm{prox}_{\lambda f(t,\cdot)}(x)$, which are all total. As discussed in Example 3.2 in \cite{Pischke2025}, by adapting arguments from \cite{RockafellarWets1998}, it is further easy to see that $\mathrm{prox}_{\lambda f}(t,\cdot)(x)$ is $\mathsf{T}$/$\mathsf{B}(\mathcal{H})$-measurable in $t$, so that $\partial f:T\to M(\mathcal{H})$ is a random monotone vector field. 
\end{example}

The notion of a random monotone vector field $A$ as defined above can be characterized by a genuine measurability condition on $A$ defined using the so-called $R$-topology, extending results of Attouch \cite{Attouch1979}, provided that $\Sigma'_x\mathcal{H}=\Sigma_x\mathcal{H}$ for all $x\in \mathcal{H}$ (recall Remark \ref{rem:maxRemark}). This topology later gives rise to the notion of $R$-convergence which is the key notion of convergence for our strong law of large numbers for random monotone vector fields.

\begin{definition}\label{RtopRem}
The $R$-topology $\tau_{R}$ on $\mathcal{M}(\mathcal{H})$ is the coarsest topology that makes all the functions $p_{x,\lambda}(A):=J^A_\lambda(x)$, given $x\in \mathcal{H}$ and $\lambda>0$, continuous.
\end{definition}

Analogously to Proposition 1.1 in \cite{Attouch1979}, we also get here that this topology is metrizable and that the resulting space is separable if $\Sigma'_x\mathcal{H}=\Sigma_x\mathcal{H}$ for all $x\in \mathcal{H}$. Before we derive this result, we however first give the following lemma, illustrating that the resolvents completely determine a monotone vector field with the surjectivity condition.

\begin{lemma}\label{resUnique}
Suppose that $\Sigma'_x\mathcal{H}=\Sigma_x\mathcal{H}$ for all $x\in \mathcal{H}$. Let $A,B\in\mathcal{M}(\mathcal{H})$ such that $J^A_\lambda(x)=J^B_\lambda(x)$ for all $\lambda>0$ and $x\in \mathcal{H}$. Then $A=B$.
\end{lemma}
\begin{proof}
We first show
\[
A=\{(J^{A}_\lambda(x),\lambda^{-1}\log_{J^{A}_\lambda(x)}x)\mid x\in \mathcal{H}, \lambda>0\}.
\]
Concretely, call the right-hand side $C$. Then $C\subseteq A$ and hence $C$ is a monotone vector field. Now, by definition we have $\lambda^{-1}\log_{J^{A}_\lambda(x)}x\in C(J^{A}_\lambda(x))$, and so by uniqueness of the resolvent, we have $J^{C}_\lambda(x)=J^{A}_\lambda(x)$ for any $x\in \mathcal{H}$ and $\lambda>0$. In particular, $C$ satisfies the surjectivity condition and hence is maximal (recall Remark \ref{rem:maxRemark}). Therefore, we have $A=C$. Now, we similarly get the above representation for $B$, and so we immediately get
\begin{align*}
A&=\{(J^{A}_\lambda(x),\lambda^{-1}\log_{J^{A}_\lambda(x)}x)\mid x\in \mathcal{H}, \lambda>0\}\\
&=\{(J^{B}_\lambda(x),\lambda^{-1}\log_{J^{B}_\lambda(x)}x)\mid x\in \mathcal{H}, \lambda>0\}=B
\end{align*}
if $J^A_\lambda(x)=J^B_\lambda(x)$ for all $\lambda>0$ and $x\in \mathcal{H}$, as claimed.
\end{proof}

We can now characterize the $R$-topology further, using a similar approach as Attouch \cite{Attouch1979}.

\begin{lemma}\label{almostPolish}
Suppose that $\Sigma'_x\mathcal{H}=\Sigma_x\mathcal{H}$ for all $x\in \mathcal{H}$. Then the $R$-topology on $\mathcal{M}(\mathcal{H})$ is metrizable with metric 
\[
d_R(A,B):=\sum_{n=1}^\infty \frac{1}{2^{n}}\min\{1,d(J^A_{r_n}(x_n),J^B_{r_n}(x_n))\},
\]
where $\{(x_n,r_n)\mid n\in\mathbb{N}\}$ is dense in $\mathcal{H}\times (0,\infty)$. Further, the space $(\mathcal{M}(\mathcal{H}),d_R)$ is separable.
\end{lemma}
\begin{proof}
It follows easily from Lemma \ref{resUnique} that $d_R$ is a metric. Now, note that $\tau_R$ is the topology induced by $d_R$. For that, note that the topology induced by $d_R$ is the coarsest topology that makes all the functions $J^A_{r_n}(x_n)$ continuous. It hence only remains to be shown that  $J^A_\lambda(x)$ is continuous in $A$ under $d_R$, given $\lambda>0$ and $x\in \mathcal{H}$. For this, suppose $d_R(A_n,A)\to 0$. Then by definition $d(J^{A_n}_{r_k}(x_k),J^A_{r_k}(x_k))\to 0$ for any $k\in\mathbb{N}$. We now first show that $J^{A_n}_{r_k}(x)$ is bounded for any $k\in\mathbb{N}$ and $x\in \mathcal{H}$. For that, let $b\geq d(x_k,x)$ and note that for large enough $n$, we get $d(J^{A}_{r_k}(x_k),J^{A_n}_{r_k}(x_k))<\varepsilon$ for this $k$. Then
\[
d(J^{A}_{r_k}(x),J^{A_n}_{r_k}(x))\leq d(J^{A}_{r_k}(x),J^{A}_{r_k}(x_k))+d(J^{A}_{r_k}(x_k),J^{A_n}_{r_k}(x_k))+d(J^{A_n}_{r_k}(x_k),J^{A_n}_{r_k}(x))<2b+\varepsilon
\]
for all such large enough $n$. Now, given $\lambda>0$ and $x\in \mathcal{H}$, take a $k\in\mathbb{N}$ such that $d(x_k,x)<\frac{\varepsilon}{9}$ and, using Lemma \ref{contRes}, such that $r_k\geq\lambda$ is close enough to $\lambda$ such that $d(J^{A}_{\lambda}(x),J^{A}_{r_k}(x))<\frac{\varepsilon}{3}$ and 
\[
d(J^{A_n}_{r_k}(x),J^{A_n}_\lambda(x))\leq 2\left( 1-\frac{\lambda}{r_k}\right) d(J^{A_n}_{r_j}(x),x)<\frac{\varepsilon}{3},
\]
where $j$ is some arbitrary index such that $r_j\geq r_k$. The last property in particular uses that $d(J^{A_n}_{r_j}(x),x)$ is bounded for varying $n$, as established previously. As above, for large enough $n$, we get that $d(J^{A}_{r_k}(x),J^{A_n}_{r_k}(x))<\frac{\varepsilon}{3}$ for this $k$, and hence obtain
\[
d(J^{A}_{\lambda}(x),J^{A_n}_{\lambda}(x))\leq d(J^{A}_{\lambda}(x),J^{A}_{r_k}(x))+d(J^{A}_{r_k}(x),J^{A_n}_{r_k}(x))+d(J^{A_n}_{r_k}(x),J^{A_n}_{\lambda}(x))<\varepsilon.
\]
To see that $(\mathcal{M}(\mathcal{H}),d_R)$ is separable, define the map $i:\mathcal{M}(\mathcal{H})\to \mathcal{H}^\mathbb{N}$ via
\[
i(A):=(J^A_{r_n}(x_n))_{n\in\mathbb{N}}.
\]
The map $i$ is now injective: Suppose that $A,B\in\mathcal{M}(\mathcal{H})$ are such that $J^A_{r_n}(x_n)=J^B_{r_n}(x_n)$ for all $n\in\mathbb{N}$. Then by Lemmas \ref{ResProp} and \ref{contRes}, we get that $J^A_{\lambda}(x)=J^B_{\lambda}(x)$ for all $\lambda>0$ and $x\in \mathcal{H}$. Lemma \ref{resUnique} now yields that $A=B$. In particular, $i$ is a bijection between $\mathcal{M}(\mathcal{H})$ and $i(\mathcal{M}(\mathcal{H}))$. Now, note that by definition of the $R$-topology, the map $i$ is a homeomorphism between $\mathcal{M}(\mathcal{H})$ and $i(\mathcal{M}(\mathcal{H}))$. As $\mathcal{H}$ is separable, so is $\mathcal{H}^\mathbb{N}$ and hence $\mathcal{M}(\mathcal{H})$ is also separable.
\end{proof}

\begin{remark}
Proposition 1.1 in \cite{Attouch1979} also establishes that $(\mathcal{M}(\mathcal{H}),d_R)$ is complete for a separable Hilbert space $\mathcal{H}$. Lifting this completeness to the present nonlinear setting is not completely trivial: While the approach from \cite{Attouch1979} carries over conceptually, it relies on the ability to associate a monotone vector field to a given nonexpansive map which is not fully available in this nonlinear context (see Proposition 4.7 in \cite{ChaipunyaKohsakaKumam2021} for some partial results in the context of the geodesic extension property). We however do not need the completeness of $(\mathcal{M}(\mathcal{H}),d_R)$ here, and so simply omit results related to this.
\end{remark}

The following result now partially extends Lemma 2.1 from \cite{Attouch1979}:

\begin{lemma}\label{randMonChar}
Suppose that $\Sigma'_x\mathcal{H}=\Sigma_x\mathcal{H}$ for all $x\in \mathcal{H}$. Let $A:T\to \mathcal{M}(\mathcal{H})$ be given. Then $A$ is a random monotone vector field if, and only if, $A$ is $\mathsf{T}$\emph{/}$\mathsf{B}(\mathcal{M}(\mathcal{H}))$-measurable.
\end{lemma}
\begin{proof}
Suppose that $A$ is $\mathsf{T}$/$\mathsf{B}(\mathcal{M}(\mathcal{H}))$-measurable. Then $J^{A(t)}_\lambda(x)=p_{x,\lambda}(A(t))$ is $\mathsf{T}$/$\mathsf{B}(\mathcal{H})$-measurable as $p_{x,\lambda}$ is continuous in $\mathcal{M}(\mathcal{H})$. Hence, $A$ is a random monotone vector field.

Conversely, suppose that each $J^{A(t)}_\lambda(x)$ is $\mathsf{T}$/$\mathsf{B}(\mathcal{H})$-measurable. Then, for an arbitrary but fixed $B\in \mathcal{M}(\mathcal{H})$, we have
\[
d_R(A(t),B):=\sum_{n=1}^\infty \frac{1}{2^{n}}\min\{1,d(J^{A(t)}_{r_n}(x_n),J^B_{r_n}(x_n))\}.
\]
As each summand is measurable, we get that also the series and hence $d_R(A(t),B)$ is $\mathsf{T}$-measurable for any $B\in\mathcal{M}(\mathcal{H})$. As $(\mathcal{M}(\mathcal{H}),d_R)$ is separable by Lemma \ref{almostPolish}, we get that $A$ is $\mathsf{T}$/$\mathsf{B}(\mathcal{M}(\mathcal{H}))$-measurable (see e.g.\ Lemma III.16 in \cite{CastaingValadier1977}).
\end{proof}

Further, we can show that for a fixed $x\in \mathcal{H}$, the map $t\mapsto A(t,x)$ is measurable as a set-valued map $T\to 2^{T_x\mathcal{H}}$, which will be useful later. For that, we first recall the following two key notions from measurable selection theory (we generally refer to \cite{AubinFrankowska2009,CastaingValadier1977} for further background): Over a complete separable metric space $(\mathcal{H},d)$ and a measurable space $(T,\mathsf{T})$, a set-valued map $\varphi:T\to 2^\mathcal{H}$ is called graph measurable if
\[
\mathrm{gra}(\varphi):=\{(t,x)\in T\times \mathcal{H}\mid x\in\varphi(t)\}\in\mathsf{T}\otimes\mathsf{B}(\mathcal{H}).
\]
Over complete $\sigma$-finite measure spaces, and if $\varphi$ has nonempty closed images, this is equivalent (see e.g.\ Theorem 8.1.4 in \cite{AubinFrankowska2009}) to the so-called weak measurability of $\varphi:T\to 2^\mathcal{H}$, that is that 
\[
\varphi^{-1}(C):=\{t\in T\mid \varphi(t)\cap C\neq\emptyset\}\in\mathsf{T}
\]
for all open sets $C\subseteq \mathcal{H}$, as well as to the so-called measurability of $\varphi$, that is that $\varphi^{-1}(C)\in\mathsf{T}$ for all closed sets $C\subseteq \mathcal{H}$. 

\begin{lemma}\label{Ameas}
Suppose that $\Sigma'_x\mathcal{H}=\Sigma_x\mathcal{H}$ for all $x\in \mathcal{H}$. Let $(T,\mathsf{T},\tau)$ be a complete probability space. Let $A:T\to\mathcal{M}(\mathcal{H})$ be a random monotone vector field and fix $x\in \mathcal{H}$. Then $t\mapsto A(t,x)$ is measurable as a set-valued map $T\to 2^{T_x\mathcal{H}}$.
\end{lemma}
\begin{proof}
Similar as in the proof of Lemma \ref{resUnique}, we have 
\[
A(t)=\{(J^{A(t)}_\lambda(y),\lambda^{-1}\log_{J^{A(t)}_\lambda(y)}y)\mid y\in \mathcal{H}, \lambda>0\}.
\]
Hence, we get
\[
A(t,x)=\{\lambda^{-1}\log_{J^{A(t)}_\lambda(y)}y\mid y\in \mathcal{H}, \lambda>0\text{ such that }J^{A(t)}_\lambda(y)=x\}.
\]
Hence, for $U\subseteq T_x\mathcal{H}$ closed, we have
\[
\{t\in T\mid A(t,x)\cap U\neq\emptyset\}=\pi_T(E_U)
\]
where $\pi_T$ is the projection onto $T$ and 
\[
E_U=\{(t,y,\lambda)\in T\times \mathcal{H}\times (0,\infty)\mid J^{A(t)}_\lambda(y)=x\text{ and }\lambda^{-1}\log_xy\in U\}.
\]
Note that $J^{A(t)}_\lambda(y)$ is jointly continuous in $(y,\lambda)\in \mathcal{H}\times (0,\infty)$ since we have
\begin{align*}
d(J^{A(t)}_{\lambda_n}(y_n),J^{A(t)}_{\lambda}(y))&\leq d(J^{A(t)}_{\lambda_n}(y_n),J^{A(t)}_{\lambda_n}(y))+d(J^{A(t)}_{\lambda_n}(y),J^{A(t)}_{\lambda}(y))\\
&\leq d(y_n,y)+d(J^{A(t)}_{\lambda_n}(y),J^{A(t)}_{\lambda}(y))\to 0
\end{align*}
as $(y_n,\lambda_n)\to (y,\lambda)$, using the nonexpansivity of the resolvent and Lemma \ref{contRes}. As $J^{A(t)}_\lambda(y)$ is also measurable in $t$, we get that $J^{A(t)}_\lambda(y)$ is a Carath\'eodory function and hence that it is jointly measurable in $(t,y,\lambda)\in T\times \mathcal{H}\times (0,\infty)$. Hence, by the inverse image theorem (see Theorem 8.2.9 in \cite{AubinFrankowska2009}), we have that
\[
t\mapsto
\{(y,\lambda)\in \mathcal{H}\times (0,\infty)\mid J^{A(t)}_\lambda(y)=x\}
\]
is measurable and hence also graph measurable. Similarly, we have that 
\[
t\mapsto \{(y,\lambda)\in \mathcal{H}\times (0,\infty)\mid \lambda^{-1}\log_xy\in U\}
\]
is measurable and hence graph measurable. Combined, we get that $E_U\in \mathsf{T}\otimes \mathsf{B}(\mathcal{H})\otimes \mathsf{B}((0,\infty))$. Hence $\pi_T(E_U)\in \mathsf{T}$ by the measurable projection theorem (see Theorem III.23 in \cite{CastaingValadier1977}). This implies that $A(\cdot,x)$ is measurable.
\end{proof}

For a random monotone vector field $A:T\to\mathcal{M}(\mathcal{H})$, we define its integral operator via
\[
\int A(t)\,d\tau:\mathcal{H}\to 2^{T\mathcal{H}},\quad x\mapsto \int A(t,x)\,d\tau(t),
\]
where the integral refers to the Aumann-Sturm integral as defined before, now on $T_x\mathcal{H}$. As already remarked in \cite{Pischke2025}, if $\mathcal{H}$ is a separable Hilbert-Hadamard space and $A$ is a random monotone vector field, then $\int A(t)\,d\tau(t)$ is immediately monotone, using that $g_x$ is affine (recall Lemma \ref{affine}).

In this paper, we now need a concept of a random monotone vector field being integrable. In \cite{Salim2023}, Salim called a random monotone operator on a Hilbert space integrable if $\int A(t)\,d\tau(t)$ is maximal (see also already \cite{BianchiHachem2016,BianchiHachemSalim2019} for previous uses of this notion). This however is used therein exclusively to guarantee that the associated resolvents are total functions via Minty's theorem \cite{Minty1962}. This equivalence is not available in this nonlinear context as discussed before,  and we hence focus on the resolvents right away and arrive at the following condition.

\begin{definition}
A random monotone vector field $A:T\to\mathcal{M}(\mathcal{H})$ is called integrable if $\int A(t)\,d\tau(t)$ satisfies the surjectivity condition.
\end{definition}

Examples of random monotone operators on Hilbert spaces which are integrable are discussed in \cite{Salim2023} (see also \cite{BianchiHachem2016}). Modelled on these discussions, we in the following discuss a few cases where similar results can be guaranteed in more general metric contexts.

\begin{example}\label{exampleInteg}
Let $(T,\mathsf{T},\tau)$ be a complete probability space and $A:T\to\mathcal{M}(\mathcal{H})$ be a random monotone vector field.
\begin{enumerate}
\item Suppose that $A$ is $L^1$-dominated, that is there exists a $g\in L^1(T,(0,\infty),\tau)$ such that $A(t,x)\neq\emptyset$ and $\sup_{y\in A(t,x)}\norm{y}_{x}\leq g(t)$ pointwise everywhere. As shown in Example 2 of \cite{Bianchi2016}, over $\mathbb{R}^d$, an $L^1$-dominated random monotone operator is integrable. This extends to finite-dimensional Hadamard manifolds as follows: By Remark 4.4 of \cite{LiLopezMartinMarquez2009}, it suffices to show that $\int A(t)\,d\tau(t)$ has full domain and is maximally monotone, and by Theorem 3.7 of \cite{LiLopezMartinMarquez2009}, this reduces to $\int A(t)\,d\tau(t)$ being monotone, upper semicontinuous and having full domain as well as having closed and convex values. Monotonicity is immediate as a Hadamard manifold is a Hilbert-Hadamard space. Further, fixing $x\in \mathcal{H}$, first recall from Lemma \ref{Ameas} that $A(\cdot,x)$ is measurable. Also, its values are closed (since $A$ is maximal).  We hence get that $\mathcal{S}^1(A(\cdot,x))\neq\emptyset$ by the Kuratowski–Ryll-Nardzewski measurable selection theorem (see Theorem 8.1.3 in \cite{AubinFrankowska2009}), using that $A$ is dominated to guarantee integrability. Hence, $\int A(t)\,d\tau(t)$ has full domain. Further, the values are convex and closed as can be immediately seen (see also \cite{Podczeck2008}). As the tangent cones are just Euclidean spaces and as the values of $A(t,\cdot)$ are closed and bounded, they are compact. Hence, the results of \cite{Yannelis1990} imply that $\int A(t)\,d\tau(t)$ is upper semicontinuous as each $A(\cdot,x)$ has measurable graph.
\item Suppose $A(t,x)=\partial f(t,x)$ is the subdifferential of a proper normal convex integrand $f:T\times \mathcal{H}\to (-\infty,+\infty]$, which is a random monotone vector field as discussed in Example \ref{ex:RMVF}. Assume that $\int f(t,\cdot)\,d\tau(t)$ is proper and lsc (it is immediately convex), where we  follow the convention that $(+\infty)-(+\infty) = +\infty$ when forming the integral. If we can interchange the integral and the subdifferential, that is if 
\[
\int \partial f(t,x)\,d\tau(t)=\partial\left( \int f(t,\cdot)\,d\tau(t)\right)(x)
\]
for all $x\in \mathcal{H}$, then $\int \partial f(t,\cdot)\,d\tau(t)$ satisfies the surjectivity condition as $\partial\left( \int f(t,\cdot)\,d\tau(t)\right)$ does (recall Example \ref{subDiffEx}). Suitable conditions which guarantee this interchange of the integral and the subdifferential on Hilbert spaces can be found in \cite{RockafellarWets1982}. We will later extend this result to separable Hilbert-Hadamard spaces where every tangent cone is isometric to a Hilbert space, covering in particular Hadamard manifolds, provided that $(T,\mathsf{T},\tau)$ is additionally separable and $f$ is $L^2$-Lipschitz and suitably integrable.
\end{enumerate}
If $(T,\mathsf{T},\tau)$ is a finite discrete measure space with $T=\{t_1,\dots,t_n\}$ as well as $\tau(\{t_i\})=w_i$, and $\mathcal{H}$ is such that each $T_x\mathcal{H}$ is isometric to a Hilbert space, then $\int A(t,x)\,d\tau(t)$ reduces to an arithmetic mean of the image sets $\sum_{i=1}^n w_i A(t_i,x)$, where the sum refers to the Minkowski sum. In particular, if $\mathcal{H}$ is a finite dimensional Hadamard manifold and all $A(t_i,\cdot)$ have full domain, $\int A(t)\,d\tau(t)$ satisfies the surjectivity condition as can be seen similarly to item (1) above: Since the sum also has full domain, by Remark 4.4 of \cite{LiLopezMartinMarquez2009} and Theorem 3.7 of \cite{LiLopezMartinMarquez2009}, $\int A(t)\,d\tau(t)$ satisfying the surjectivity condition reduces to $\int A(t)\,d\tau(t)$ being monotone, upper semicontinuous and having closed and convex values which we can establish in this case. First, in the present context, $\int A(t)\,d\tau(t)$ is clearly monotone and its values are convex as the Minkowski sum preserves convexity. By Lemma 3.6 in \cite{LiLopezMartinMarquez2009}, the $A(t_i,\cdot)$ are locally bounded and so, as they are also closed, we get that $\int A(t)\,d\tau(t)$ has closed values since the tangent cones of $\mathcal{H}$ are finite dimensional Hilbert spaces. Also, by the same argument as given in item (1) above, we get that $\int A(t)\,d\tau(t)$ is upper semicontinuous.
\end{example}

\section{A strong law of large numbers for random monotone vector fields}\label{sec:SLLNVec}

We now turn to the second main result of this paper, an extension of the previous $L^1$ strong law of large numbers for random variables on Hadamard spaces to integrable random monotone vector fields. For this, unless said otherwise, let $(\mathcal{H},d)$ now be a separable Hilbert-Hadamard space (and recall Remark \ref{rem:tangentSepAndMeas}).

Again, we require a notion of finite mean, and for this we extend the previously discussed inductive mean of Sturm from points to sets.

\begin{definition}[Inductive mean of sets]
The inductive mean of $A_1,\dots,A_n\subseteq \mathcal{H}$ is defined as follows:
\[
\frac{1}{n}\underset{i=1,\dots,n}{\dsum} A_i:=\left\{\frac{1}{n}\underset{i=1,\dots,n}{\dsum} x_i\mid x_i\in A_i\text{ for }i=1,\dots,n\right\}.
\]
\end{definition}

If $A_i:\mathcal{H}\to 2^{T\mathcal{H}}$ are monotone vector fields, we define inductive mean vector field via
\[
\frac{1}{n}\underset{i=1,\dots,n}{\dsum} A_i:\mathcal{H}\to 2^{T\mathcal{H}},\quad x\mapsto \frac{1}{n}\underset{i=1,\dots,n}{\dsum} A_i(x).
\]
In particular, this vector field is monotone in a separable Hilbert-Hadamard space $\mathcal{H}$, using that $g_x$ is affine (recall Lemma \ref{affine}) similar as with $\int A(t)\,d\tau$.

Before we can derive this second strong law of large numbers, we require an intermediate step. In a Hadamard manifold $\mathcal{H}$, using that $\nabla_x\frac{1}{2}d^2(x,y)=-\log_xy$ (as follows from the first variation formula, see \cite{Sakai1996}) and that $\frac{1}{2}d^2(x,y)$ is strongly convex, it can be shown that $-\log_x$ is strongly monotone, that is
\[
g_z(\log_zx,\log_zy)\geq -g_y(\log_yx,\log_yz)+d^2(z,y)
\]
for all $x,y,z\in \mathcal{H}$. Our proof of the strong law of large numbers for integrable random monotone vector fields relies on this property and while the conclusion of strong monotonicity as formulated above does make sense in the context of Hadamard spaces, there are two immediate issues with the above approach: First, a Hadamard space does not carry differential structure as is, and second, the expression $-\log_x$ is not well-defined as the space $T_x\mathcal{H}$ is not linear but just a cone.

We hence rely on a different approach here. Concretely, we follow the argument used in Proposition 2.9 from \cite{Pischke2025} to show that $\partial f$ is strongly monotone whenever $f$ is strongly convex, but utilize $-\frac{1}{2}d^2(x,y)$ in place of $f$.

\begin{lemma}\label{negLogMon}
Let $(\mathcal{H},d)$ be a Hadamard space. For any $x,y,z\in \mathcal{H}$:
\[
g_z(\log_zx,\log_zy)\geq -g_y(\log_yx,\log_yz)+d^2(z,y)
\]
\end{lemma}
\begin{proof}
For $f(a):=-\frac{1}{2}d^2(a,x)$, we get for any $t\in [0,1]$ and any $a\in \mathcal{H}$ that
\[
(1-t)f(a)+tf(b)+t(1-t)\frac{1}{2}d^2(a,b)\leq f((1-t)a\oplus t b).
\]
Lemma \ref{tangentCat0} yields
\begin{align*}
2g_a(\log_a x,\log_a y)&\geq d^2(a,x)+d^2(a,y)-d^2(x,y)\\
&\geq d^2(a,x)-d^2(x,y)
\end{align*}
and so
\[
f(y)\leq f(a)+g_a(\log_a x,\log_a y)
\]
for all $y\in \mathcal{H}$. In particular, we have
\[
f((1-t)a\oplus t b)\leq f(a)+g_a(\log_a x,\log_a ((1-t)a\oplus t b)),
\]
so that combined, we have
\[
(1-t)f(a)+tf(b)+t(1-t)\frac{1}{2}d^2(a,b)\leq f(a)+g_a(\log_a x,\log_a ((1-t)a\oplus t b)).
\]
Hence
\[
f(b)-f(a)\leq \frac{g_a(\log_a x,\log_a ((1-t)a\oplus t b))}{t}-(1-t)\frac{1}{2}d^2(a,b)
\]
for all $t\in (0,1]$. Similarly, we get
\[
f(a)-f(b)\leq \frac{g_b(\log_b x,\log_b ((1-t)b\oplus t a))}{t}-(1-t)\frac{1}{2}d^2(b,a).
\]
Using $g_a(\log_a x,\log_a ((1-t)a\oplus t b))/t=g_a(\log_ax,\log_ab)$ (and similarly for $b$), as established in the proof of Proposition 2.9 in \cite{Pischke2025}, we have
\[
g_a(\log_ax,\log_ab)-(1-t)d^2(a,b)\geq -g_b(\log_bx,\log_ba).
\]
Taking $t\to 0$, we get
\[
g_a(\log_ax,\log_ab)\geq -g_b(\log_bx,\log_ba)+d^2(a,b)
\]
as claimed.
\end{proof}

We can now establish our strong law of large numbers for integrable random monotone vector fields over Hilbert-Hadamard spaces. Similar to \cite{Salim2023}, the mode of convergence for the random monotone vector fields will be that of $R$-convergence, as already mentioned before. Concretely, the $R$-topology discussed previously in Definition \ref{RtopRem} gives rise to the following notion of $R$-convergence: Given a sequence $(A_n)\subseteq\mathcal{M}(\mathcal{H})$ together with $A\in \mathcal{M}(\mathcal{H})$, we say that $(A_n)$ $R$-converges to $A$, written $A_n\to^R A$, if
\[
J^{A_n}_\lambda(x)\to J^{A}_\lambda(x)
\]
for all $x\in \mathcal{H}$ and $\lambda>0$. In particular, if $A_n,A:\Omega\to \mathcal{M}(\mathcal{H})$ are random monotone vector fields, then $A_n\to^R A$ holds $\PP$-a.s.\ if, and only if, there exists a set $\Omega'\subseteq \Omega$ of measure one such that $J^{A_n(\omega)}_\lambda(x)\to J^{A(\omega)}_\lambda(x)$ for all $\omega\in \Omega'$, $x\in \mathcal{H}$ and $\lambda>0$.

We now turn to the second main result of this paper:

\begin{theorem}\label{L1LawVector}
Let $(\Omega,\mathsf{F},\PP)$ and $(E,\mathsf{E},\mu)$ be probability spaces and let $(\mathcal{H},d)$ be a separable Hilbert-Hadamard space. Further, fix an i.i.d.\ sequence $(\xi_i)$ of random variables $\xi_i:\Omega\to E$ with distribution $\mu$. Let $A:E\to \mathcal{M}(\mathcal{H})$ be an integrable random monotone vector field such that $\frac{1}{n}\underset{i=1,\dots,n}{\dsum} A(\xi_i)$ satisfies the surjectivity condition $\PP$-a.s., for all $n\in\mathbb{N}$. Then
\[
\frac{1}{n}\underset{i=1,\dots,n}{\dsum} A(\xi_i)\to^R \int A(e)\, d\mu(e)\quad\PP\text{-a.s.}
\]
as $n\to\infty$.
\end{theorem}
\begin{proof}
For brevity, write 
\[
\underline{A}_n:=\frac{1}{n}\underset{i=1,\dots,n}{\dsum} A(\xi_i)\text{ and }\underline{A}:=\int A(e)\, d\mu(e).
\]
Given any $x\in \mathcal{H}$ and $\lambda>0$, write $y:=J^{\underline{A}}_\lambda(x)$ and $z:=\frac{1}{\lambda}\log_y x\in \underline{A}(y)$. As $z\in \underline{A}(y)$, there is an $\alpha\in\mathcal{S}^1(A(\cdot,y))$ such that $z=\int\alpha\,d\mu$. Define
\[
z_n:=\frac{1}{n}\underset{i=1,\dots,n}{\dsum} \alpha(\xi_i)
\]
and note that $z_n(\omega)\in\underline{A}_n(\omega,y)$. Then by Corollary \ref{L1LawHad} on $T_y\mathcal{H}$, we have 
\[
z_n\to z\quad \PP\text{-a.s.}
\]
as $n\to\infty$. Now, define $y_n(\omega):=J^{\underline{A}_n(\omega)}_\lambda(x)$ and $u_n(\omega):=\frac{1}{\lambda}\log_{y_n(\omega)}x\in \underline{A}_n(\omega,y_n(\omega))$. As $ \underline{A}_n(\omega,\cdot)$ is monotone for any $\omega$, we get
\[
g_{y_n}(u_n,\log_{y_n}y)\leq -g_y(z_n,\log_yy_n)
\]
almost surely. Lemma \ref{negLogMon} yields that
\[
g_{y_n}(\log_{y_n}x,\log_{y_n}y)\geq -g_y(\log_yx,\log_y{y_n})+d^2({y_n},y),
\]
and so
\[
g_{y_n}(u_n,\log_{y_n}y)\geq -g_y(z,\log_y{y_n})+\frac{1}{\lambda}d^2({y_n},y).
\]
Combined, we have 
\begin{align*}
d^2(y_n,y)&\leq \lambda\left(g_{y_n}(u_n,\log_{y_n}y) +g_y(z,\log_y{y_n})\right)\\
&\leq \lambda\left(g_y(z,\log_y{y_n})-g_y(z_n,\log_yy_n)\right)\\
&\leq \lambda\norm{\log_yy_n} d_y(z,z_n)\\
&= \lambda d(y,y_n) d_y(z,z_n),
\end{align*}
where the third inequality follows from Lemma \ref{nonexpCat0}. Hence, we have
\[
d(J^{\underline{A}_n(\omega)}_\lambda(x),J^{\underline{A}}_\lambda(x))=d(y_n(\omega),y)\leq \lambda d_y(z,z_n(\omega))\to 0
\]
for almost all $\omega$ as $n\to\infty$. This shows that for any $x\in \mathcal{H}$ and $\lambda>0$, there exists a set $\Omega_{x,\lambda}$ of measure one such that 
\[
J^{\underline{A}_n(\omega)}_\lambda(x)\to J^{\underline{A}}_\lambda(x)
\]
for all $\omega\in\Omega_{x,\lambda}$. Using that $\mathcal{H}$ is separable, let $Z\subseteq \mathcal{H}$ be a countable dense subset. Define 
\[
\Omega':=\bigcap_{z\in Z,r\in\mathbb{Q}_{>0}}\Omega_{z,r},
\]
which has measure one as well. Now, let $\omega\in\Omega'$ as well as $x\in \mathcal{H}$ and $\lambda>0$ be arbitrary. Let $\varepsilon\in (0,1]$ be arbitrary and let $z\in Z$ be such that $d(x,z)<\frac{\varepsilon}{6}$ and let $r\in\mathbb{Q}_{>0}$ be such that $\lambda\leq r\leq \min\{\lambda/(1-\varepsilon/24b),\lceil 2\lambda\rceil\}$ for $b\geq 1$ such that $b\geq d(J^A_{\lceil{2\lambda\rceil}}(z),z)$. Then, using the nonexpansivity of the resolvents $J^{\underline{A}}_\lambda$ and $J^{\underline{A}_n}_\lambda$, we get
\begin{align*}
d(J^{\underline{A}_n(\omega)}_\lambda(x),J^{\underline{A}}_\lambda(x))&\leq d(J^{\underline{A}_n(\omega)}_\lambda(x),J^{\underline{A}_n(\omega)}_\lambda(z))+d(J^{\underline{A}_n(\omega)}_\lambda(z),J^{\underline{A}}_\lambda(z))+d(J^{\underline{A}}_\lambda(x),J^{\underline{A}}_\lambda(z))\\
&\leq 2d(x,z)+d(J^{\underline{A}_n(\omega)}_\lambda(z),J^{\underline{A}}_\lambda(z))\\
&\leq 2d(x,z)+d(J^{\underline{A}_n(\omega)}_\lambda(z),J^{\underline{A}_n(\omega)}_r(z))+d(J^{\underline{A}_n(\omega)}_r(z),J^{\underline{A}}_r(z))+d(J^{\underline{A}}_r(z),J^{\underline{A}}_\lambda(z)).
\end{align*}
Now, note that 
\[
d(J^{\underline{A}}_r(z),J^{\underline{A}}_\lambda(z))\leq 2\left( 1-\frac{\lambda}{r}\right) d(J^A_{\lceil{2\lambda\rceil}}(z),z)\leq 2\left( 1-\frac{\lambda}{r}\right) b
\]
using (the proof of) Lemma \ref{contRes}. Similarly, we have
\begin{align*}
d(J^{\underline{A}_n(\omega)}_\lambda(z),J^{\underline{A}_n(\omega)}_r(z))&\leq 2\left( 1-\frac{\lambda}{r}\right) d(J^{\underline{A}_n(\omega)}_{\lceil 2\lambda\rceil}(z),z)\\
&\leq 2\left( 1-\frac{\lambda}{r}\right) d(J^{\underline{A}_n(\omega)}_{\lceil 2\lambda\rceil}(z),J^{\underline{A}}_{\lceil 2\lambda\rceil}(z))+2\left( 1-\frac{\lambda}{r}\right)d(J^{\underline{A}}_{\lceil 2\lambda\rceil}(z),z)\\
&\leq 2d(J^{\underline{A}_n(\omega)}_{\lceil 2\lambda\rceil}(z),J^{\underline{A}}_{\lceil 2\lambda\rceil}(z))+2\left( 1-\frac{\lambda}{r}\right) b.
\end{align*}
Combined, we have
\begin{align*}
d(J^{\underline{A}_n(\omega)}_\lambda(x),J^{\underline{A}}_\lambda(x))&\leq 2d(x,z)+2d(J^{\underline{A}_n(\omega)}_{\lceil 2\lambda\rceil}(z),J^{\underline{A}}_{\lceil 2\lambda\rceil}(z))+4\left( 1-\frac{\lambda}{r}\right) b+d(J^{\underline{A}_n(\omega)}_r(z),J^{\underline{A}}_r(z))\\
&\leq \frac{3\varepsilon}{6}+2d(J^{\underline{A}_n(\omega)}_{\lceil 2\lambda\rceil}(z),J^{\underline{A}}_{\lceil 2\lambda\rceil}(z))+d(J^{\underline{A}_n(\omega)}_r(z),J^{\underline{A}}_r(z)).
\end{align*}
Using $\omega\in\Omega'\subseteq \Omega_{z,r}\cap \Omega_{z,\lceil 2\lambda\rceil}$, there is an $m$ such that
\[
d(J^{\underline{A}_n(\omega)}_r(z),J^{\underline{A}}_r(z)),d(J^{\underline{A}_n(\omega)}_{\lceil 2\lambda\rceil}(z),J^{\underline{A}}_{\lceil 2\lambda\rceil}(z))\leq \frac{\varepsilon}{6}
\]
for all $n\geq m$. Hence
\[
d(J^{\underline{A}_n(\omega)}_\lambda(x),J^{\underline{A}}_\lambda(x))\leq \varepsilon\text{ for all }n\geq m.
\]
As $\varepsilon\in (0,1]$ was arbitrary, we get that $J^{\underline{A}_n(\omega)}_\lambda(x)$ converges to $J^{\underline{A}}_\lambda(x)$. This yields the result.
\end{proof}

\begin{remark}\label{vectorLawQuant}
Also the strong law for random monotone vector fields in the above Theorem \ref{L1LawVector} can be made quantitative by means of a concentration inequality. At first, note that by following the proof of Theorem \ref{L1LawVector}, we get
\[
d(J^{\underline{A}_n(\omega)}_\lambda(x),J^{\underline{A}}_\lambda(x))\leq \lambda d_y(z,z_n(\omega))
\]
for almost all $\omega$. Theorem \ref{L1LawHadQuant} can then be employed to provide a quantitative rendering of $z_n\to z$ $\PP$-a.s.: For any $\varepsilon>0$ and any $m\in\mathbb{N}$:
\[
\PP\left(\sup_{n\geq m} d\left(z_n,z\right)>\varepsilon\right)\leq \frac{6}{\varepsilon}\left(\frac{\sqrt{\EE[\norm{\alpha(\xi_1)}_y]}}{m^{1/4}}+U_\PP(m^{1/8})\right)+\frac{72}{\min\{\varepsilon^2,1\}m^{1/4}}
\]
where $U_\PP(a):=\EE[\norm{\alpha(\xi_1)}_y\mathbf{1}_{\norm{\alpha(\xi_1)}_y\geq a}]$ for $\alpha\in\mathcal{S}^1(A(\cdot,y))$ such that $z=\int\alpha\,d\mu$ and $y=J^{\underline{A}}_\lambda(x)$. Hence, we in particular have 
\[
\PP\left(\sup_{n\geq m} d(J^{\underline{A}_n}_\lambda(x),J^{\underline{A}}_\lambda(x))>\varepsilon\right)\leq \frac{6}{\varepsilon/\lambda}\left(\frac{\sqrt{\EE[\norm{\alpha(\xi_1)}_y]}}{m^{1/4}}+U_\PP(m^{1/8})\right)+\frac{72}{\min\{\varepsilon^2/\lambda^2,1\}m^{1/4}}
\]
for all $\varepsilon>0$ and $m\in\mathbb{N}$. If $A$ is $L^1$-uniformly majorized (in the spirit of \cite{Pischke2024}), that is $\sup_{y\in A(e,x)}\norm{y}_{x}\leq g(e,b)$ almost everywhere for all $b>0$ and all $x\in \mathcal{H}$ with $d(o,x)\leq b$, where $g(\cdot,b)\in L^1(E,(0,\infty),\mu)$ for all $b>0$, then we can estimate $\sqrt{\EE[\norm{\alpha(\xi_1)}_y]}\leq \sqrt{\EE[g(\cdot,c)]}$ as well as
\[
U_\PP(a)=\EE[\norm{\alpha(\xi_1)}_y\mathbf{1}_{\norm{\alpha(\xi_1)}_y\geq a}]\leq \EE[g(\cdot,c)\mathbf{1}_{g(\cdot,c)\geq a}],
\]
where $c\geq d(o,y)$. In particular, if $b'\geq d(o,J^{\underline{A}}_\lambda(o))$ and $b\geq d(o,x)$, then
\[
d(o,y)\leq d(o,J^{\underline{A}}_\lambda(o))+d(J^{\underline{A}}_\lambda(o),J^{\underline{A}}_\lambda(x))\leq b'+b
\]
so that we can take $c=b'+b$ and so that the above bound holds uniformly for all $x\in \mathcal{H}$ in terms of such $b'$ and $b$ as well as $\lambda>0$. In particular, if $A$ is $L^1$-dominated, that is $\sup_{y\in A(e,x)}\norm{y}_{x}\leq g(e)$ almost everywhere and $g(\cdot)\in L^1(E,(0,\infty),\mu)$ for all $x\in \mathcal{H}$, then the above holds uniformly for all $x\in \mathcal{H}$. We do not spell this out here further.
\end{remark}

\section{On the interchange of integration and subdifferentiation}\label{interchange}

Let $(T,\mathsf{T},\tau)$ be a complete probability space and $(\mathcal{H},d)$ be a separable Hadamard space, and let $f:T\times \mathcal{H}\to (-\infty,+\infty]$ be a proper normal convex integrand (recall \cite{RockafellarWets1998} and Example \ref{ex:RMVF}). The goal of this section is to investigate in what way the interchange between subdifferentiation and integration in the form of
\[
\int \partial f(t,x)\,d\tau(t)=\partial\left(\int f(t,\cdot)\,d\tau(t)\right)(x)\text{ for all }x\in \mathcal{H}
\]
holds true in separable Hilbert-Hadamard spaces. Results of this type have a longer history. One of the strongest instances was obtained by Rockafellar and Wets \cite{RockafellarWets1982}, who actually proved a stronger relation on the interchange of the conditional expectation and the subdifferential of $F(x):=\int f(t,x(t))\,d\tau(t)$ over $x\in L^2(T,\mathcal{H},\tau)$ and $\mathcal{H}=\mathbb{R}^d$ (see also \cite{Bismut1973} for similar results). A result such as the above for suitable functions $f$ over $\mathbb{R}^d$ was first derived by Ioffe and Tikhomirov \cite{IoffeTikhomirov1969}, and later extended and modified by Levin \cite{Levin1970}, Hiriart-Urruty \cite{HiriartUrruty1976} and Thibault \cite{Thibault1981}, among others.

Throughout, we now write $\underline{f}(x):=\int f(t,x)\,d\tau(t)$ as well as $\underline{\partial f}(x):=\int \partial f(t,x)\,d\tau(t)$, and as above set $F(x):=\int f(t,x(t))\,d\tau(t)$ for $x\in L^2(T,\mathcal{H},\tau)$. Here, we use the convention $(+\infty)-(+\infty) = +\infty$ when forming the integral. As a first observation, we note that over separable Hilbert-Hadamard spaces, one direction always holds true (as already observed in Example 3.2 of \cite{Pischke2025}):

\begin{lemma}\label{firstInterchange}
Let $(\mathcal{H},d)$ be a separable Hilbert-Hadamard space and let $f:T\times \mathcal{H}\to (-\infty,+\infty]$ be a proper normal convex integrand. For any $x\in \mathrm{dom}\underline{f}$: $\underline{\partial f}(x)\subseteq\partial \underline{f}(x)$.
\end{lemma}
\begin{proof}
Let $u\in L^1(T,T_x\mathcal{H},\tau)$ be a selection of $\partial f$, that is for all $t\in T$ and all $y\in \mathcal{H}$:
\[
f(t,y)\geq f(t,x)+g_x(u(t),\log_xy).
\]
Integrating and using Lemma \ref{intEx} yields
\[
\underline{f}(y)\geq \underline{f}(x)+g_x\left(\int u\,d\tau,\log_xy\right)
\]
for all $y\in \mathcal{H}$, so that $\int u\,d\tau\in\partial \underline{f}(x)$.
\end{proof}

Unfortunately, the converse does not hold true over general Hilbert-Hadamard spaces, as the following construction shows:

\begin{example}\label{ExNoHH}
Consider $\mathcal{H}:=[0,\infty)$. This is a non-empty, closed and convex subset of a Hilbert-Hadamard space, and so itself a Hilbert-Hadamard space. In particular, we have $T_0 \mathcal{H}=[0,\infty)$ and $\log_0 x=x$ for all $x\in \mathcal{H}:=[0,\infty)$. Consider the complete probability space on $T:=\{1,2\}$ with measure $\tau$ induced by $\tau(\{1\}):=\tau(\{2\}):=\frac{1}{2}$, and define the function $f:T\times \mathcal{H}\to\mathbb{R}$ by
\[
f(1,x):=-x\text{ and }f(2,x):=2x\text{ for all }x\in \mathcal{H}:=[0,\infty).
\]
Clearly $f(1,\cdot)$ and $f(2,\cdot)$ are proper, convex and continuous. Further, $f$ is clearly measurable. We now have
\[
\underline{f}(x)=\frac{1}{2}f(1,x)+\frac{1}{2}f(2,x)=\frac{x}{2}.
\]
It can be easily seen that $\partial\underline{f}(0)=[0,\frac{1}{2}]$. Now, it holds that $u\in\partial f(1,0)$ if, and only if, $u\in T_0\mathcal{H}$ and $u\leq -1$, but since $u\in T_0\mathcal{H}$ implies $u\geq 0$, we get $\partial f(1,0)=\emptyset$. Hence, we have
\[
\underline{\partial f}(0)=\frac{1}{2}\partial f(1,0)+\frac{1}{2}\partial f(2,0)=\emptyset
\]
which shows that $\partial\underline{f}(0)\not\subseteq \underline{\partial f}(0)$.
\end{example}

A key problem in the above example was that $T_0\mathcal{H}$, while flat, is not a linear space, so that $u\in T_0\mathcal{H}$ contradicted $u\leq -1$.

Indeed, once the linearity of the tangent cones is imposed, we are actually able to establish the converse to the above result (if $f$ is additionally Lipschitz). We first give these spaces a name:

\begin{definition}
We call a Hadamard space $\mathcal{H}$ a full Hilbert-Hadamard space if each tangent cone $T_x\mathcal{H}$ is isometric to a Hilbert space.
\end{definition}

Hilbert spaces and Hadamard manifolds as well as their infinite-dimensional generalizations are not only Hilbert-Hadamard spaces but also full Hilbert-Hadamard spaces. Further, Example \ref{HilHadEx} shows that also full Hilbert-Hadamard spaces are not restricted to the settings of manifolds. Before moving on, we note the following result relating Hilbert-Hadamard spaces to full Hilbert-Hadamard spaces. Say that $\mathcal{H}$ has the geodesic extension property if every geodesic in $\mathcal{H}$ can be extended to an isometry $\mathbb{R}\to \mathcal{H}$.
 
\begin{lemma}\label{HHtofullHH}
If $\mathcal{H}$ is a Hilbert-Hadamard space and it satisfies the geodesic extension property, then $\mathcal{H}$ is a full Hilbert-Hadamard space.
\end{lemma}
\begin{proof}
Fix $x\in \mathcal{H}$. By assumption (recall Remark \ref{flatRem}), $T_x\mathcal{H}$ isometrically embeds into a Hilbert space $H$, say via a mapping $\iota:T_x\mathcal{H}\to H$. Using a shift, we may assume that $\iota(0_x)=0$. We now show that $\iota(T_x\mathcal{H})$ is closed under additive inverses. For this, let $\gamma\in \Sigma'_x\mathcal{H}$ and let $-\gamma\in \Sigma'_x\mathcal{H}$ be an opposite geodesic. Note that $\norm{\gamma}_x=\norm{-\gamma}_x$ as well as $\angle_x(\gamma,-\gamma)=\pi$ so that
\[
d_x(\gamma,-\gamma)=\sqrt{2-2\cos\angle_x(\gamma,-\gamma)}=\sqrt{4}=2\norm{\gamma}_x.
\]
As $\iota$ is an isometry, we get $\norm{\iota(\gamma)}_H=\norm{\iota(-\gamma)}_H$ and $\norm{\iota(\gamma)-\iota(-\gamma)}_H=2\norm{\iota(\gamma)}_H$, where $\norm{\cdot}_H$ is the norm on $H$. As we hence have an equality in the triangle inequality in $H$, we get that $\iota(-\gamma)=-\iota(\gamma)$. As $\Sigma'_x\mathcal{H}$ is dense in $\Sigma_x\mathcal{H}$, the result extends to $\Sigma_x\mathcal{H}$ by continuity and then it further extends to $T_x\mathcal{H}$, the cone over $\Sigma_x\mathcal{H}$, where we in particular use that $\iota(tv)=t\iota(v)$, which holds since $\norm{\iota(v)-\iota(tv)}_H=d_x(v,tv)=\norm{v}_x\sqrt{(t-1)^2}=\norm{v}_x\vert t-1\vert=\norm{\iota(v)}_H\vert t-1\vert$. Combined, we get that $\iota(T_x\mathcal{H})$ is a nonempty closed and convex subset of $H$ containing $0$ and which is closed under non-negative scalar multiplication and additive inverses. Hence, $\iota(T_x\mathcal{H})$ is linear.
\end{proof}

To now establish the converse of the above Lemma \ref{firstInterchange} over full Hilbert-Hadamard spaces (and suitable Lipschitz functions), we broadly follow a similar approach as used in \cite{HiriartUrruty1976} (which in turn relies on the exactness result from Valadier \cite{Valadier1970}), together with suitably adapted variants of a variety of results due to Castaing and Valadier \cite{CastaingValadier1977} highlighted throughout. 

In particular, we rely on the conjugate of $F$, which we introduce following the work \cite{BergmannHerzogSilvaLouzeiroTenbrinckVidalNunez2021,SilvaLouzeiroBergmannHerzog2022} (see also \cite{BentoCruzNetoMelo2023,Hirai2024} for related works using Busemann functions). Generally, for a Hadamard space $\mathcal{H}$ and a proper $h:\mathcal{H}\to (-\infty,+\infty]$ as well as $p\in \mathcal{H}$, define $h^*_p:T_p\mathcal{H}\to (-\infty,+\infty]$ via
\[
h^*_p(u):=\sup_{q\in \mathcal{H}}\{g_p(u,\log_pq)-h(q)\}.
\]
We begin with some immediate lemmas.

\begin{lemma}\label{conjugateProp}
Let $\mathcal{H}$ be a Hadamard space and let $h:\mathcal{H}\to (-\infty,+\infty]$ be proper. For any $p\in \mathcal{H}$ and $u\in T_p\mathcal{H}$: 
\begin{enumerate}
\item $h^*_p(u)+h(p)\geq 0$,
\item $u\in\partial h(p)$ if, and only if, $h^*_p(u)+h(p)= 0$.
\end{enumerate}
\end{lemma}
\begin{proof}
\begin{enumerate}
\item If $h(p)=+\infty$, the result is immediate. Otherwise, we have that
\[
h^*_p(u)+h(p)\geq g_p(u,\log_pp)-h(p) + h(p)=0.
\]
\item Suppose $u\in\partial h(p)$, that is
\[
h(q)\geq h(p)+g_p(u,\log_pq)\text{ for all }q\in \mathcal{H}.
\]
Re-arranged, we get
\[
-h(p)\geq g_p(u,\log_pq)-h(q)\text{ for all }q\in \mathcal{H}.
\]
Taking the supremum over $q$, we get $-h(p)\geq h^*_p(u)$, that is $h^*_p(u)+h(p)\leq 0$. By (1), we get $h^*_p(u)+h(p)= 0$.

Conversely, if $h^*_p(u)+h(p)= 0$, then $h^*_p(u)+h(p)\leq 0$ and so
\[
-h(p)\geq h^*_p(u)\geq g_p(u,\log_pq)-h(q)\text{ for all }q\in \mathcal{H}.
\]
Re-arranging yields that $u\in \partial h(p)$.
\end{enumerate}
\end{proof}

As discussed in Section \ref{sec:prelim}, $L^2(T,\mathcal{H},\tau)$ is a separable Hilbert-Hadamard space whenever $\mathcal{H}$ is so and $T$ is separable, and the tangent cone $T_pL^2(T,\mathcal{H},\tau)$ of $L^2(T,\mathcal{H},\tau)$ at a point $p\in \mathcal{H}$ (identified with its constant map) is (isometric to) $L^2(T,T_p\mathcal{H},\tau)$. In particular, the pseudo-Riemannian metric takes the form of
\[
g_p(u,v)=\int g_p(u(t),v(t))\,d\tau(t).
\]
Hence, we have
\[
F^*_p(v):=\sup_{q\in L^2(T,\mathcal{H},\tau)}\{g_p(v,\log_pq)-F(q)\}
\]
for $p\in \mathcal{H}$ and $v\in L^2(T,T_p\mathcal{H},\tau)$.

We now show that $F^*_p$ can be represented as the integral over the conjugate $f^*_p$. For that, we follow the approach taken in \cite{CastaingValadier1977} for a similar result. At first, we show that the latter is a normal convex integrand over Hilbert-Hadamard spaces, which is an adaptation of Corollary VII-2 in \cite{CastaingValadier1977}.

\begin{lemma}
Let $\mathcal{H}$ be a separable Hilbert-Hadamard space and let $f:T\times \mathcal{H}\to (-\infty,+\infty]$ be a proper normal integrand. For any $p\in \mathcal{H}$, the conjugate $f^*_p:T\times T_p\mathcal{H}\to (-\infty,+\infty]$ defined by
\[
f^*_p(t,u):=\sup_{q\in \mathcal{H}}\{g_p(u,\log_pq)-f(t,q)\}
\]
is a normal convex integrand.
\end{lemma}
\begin{proof}
First, note that by Lemma \ref{affine}, $f^*_p(t,\cdot)$ is a supremum of convex and lsc functions and hence is convex and lsc. As $T$ is complete, it suffices to show that $f^*_p$ is $\mathsf{T}\otimes\mathsf{B}(T_p\mathcal{H})$-measurable. To see this, note that
\[
f^*_p(t,u)=\sup_{(q,r)\in \mathrm{epi}f(t,\cdot)}\{g_p(u,\log_pq)-r\}.
\]
As $\mathrm{epi}f(t,\cdot)$ is measurable, there is a sequence $(q_n,r_n)$ of measurable selections such that, for any $t\in T$, the sequence $(q_n(t),r_n(t))$ is dense in $\mathrm{epi}f(t,\cdot)$. Hence, we have
\[
f^*_p(t,u)=\sup_{n\in\mathbb{N}}\{g_p(u,\log_pq_n(t))-r_n(t)\},
\]
using that $g_p$ and $\log_p$ are continuous. Now, $g_p(u,\log_pq_n(t))$ is a Carath\'eodory function and hence is $\mathsf{T}\otimes\mathsf{B}(T_p\mathcal{H})$-measurable. Therefore, $f^*_p$ is a supremum of $\mathsf{T}\otimes\mathsf{B}(T_p\mathcal{H})$-measurable functions and hence is $\mathsf{T}\otimes\mathsf{B}(T_p\mathcal{H})$-measurable.
\end{proof}

Further, we need the following lemma, which is an adaptation of Proposition VII-4 in \cite{CastaingValadier1977}.

\begin{lemma}\label{Radon}
Let $\mathcal{H}$ be a complete separable metric space and $s:T\to \mathcal{H}$ be measurable. Then there exists an increasing sequence of measurable sets $T_n$ such that $\tau(T\setminus \bigcup_{n\in\mathbb{N}} T_n)=0$ and $\overline{s(T_n)}$ is compact for all $n\in\mathbb{N}$.
\end{lemma}
\begin{proof}
Consider the push-forward measure $\tau_s$ on $\mathsf{B}(\mathcal{H})$. As $\mathcal{H}$ is a Polish space, $\tau_s$ is a Radon measure (as follows from Prohorov's theorem, see Theorem 13.29 in \cite{Klenke2020}). Therefore, there exists an increasing sequence of compact sets $K_n\subseteq \mathcal{H}$ such that $\tau_s(\bigcup_{n\in\mathbb{N}} K_n)=\tau_s(\mathcal{H})$. Define $T_n:=s^{-1}(K_n)$. This gives the claim.
\end{proof}

Lastly, we require another result from measurable selection theory, stating the graph-measurability of set-valued maps induced by inequalities (extending Proposition 2I in \cite{Rockafellar1976}):

\begin{lemma}\label{epiLem}
Let $\mathcal{H}$ be a complete separable metric space and let $f:T\times \mathcal{H}\to (-\infty,+\infty]$ be a normal integrand and $\alpha$ be measurable. Then 
\[
\Gamma(t):=\{x\in \mathcal{H}\mid f(t,x)\leq \alpha(t)\}
\]
is closed-valued and graph measurable.
\end{lemma}
\begin{proof}
That $\Gamma$ is closed-valued follows immediately as $f$ is lsc in its right argument. We have 
\begin{align*}
\mathrm{gra}\Gamma&=\{(t,x)\in T\times \mathcal{H}\mid f(t,x)\leq \alpha(t)\}\\
&=\{(t,x)\in T\times \mathcal{H}\mid (x,\alpha(t))\in \mathrm{epi}f(t,\cdot)\}\\
&=\varphi^{-1}(\mathrm{gra}(\mathrm{epi}f))
\end{align*}
for $\varphi(t,x)=(t,x,\alpha(t))$, which is clearly measurable. Hence $\mathrm{gra}\Gamma\in\mathsf{T}\otimes \mathsf{B}(\mathcal{H})$ since $\mathrm{gra}(\mathrm{epi}f)\in\mathsf{T}\otimes\mathsf{B}(\mathcal{H})\otimes\mathsf{B}(\mathbb{R})$.
\end{proof}

We now can give the respective result on the conjugate $F^*_p$, where we follow the approach taken in the proof of Theorem VII-7 in \cite{CastaingValadier1977}.

\begin{lemma}\label{conjugateCharac}
Let $T$ be separable and let $\mathcal{H}$ be a separable Hilbert-Hadamard space, let $f:T\times \mathcal{H}\to (-\infty,+\infty]$ be a proper normal convex integrand and assume that $F$ is proper and lsc. For any $p\in \mathcal{H}$ and any $v\in L^2(T,T_p\mathcal{H},\tau)$:
\[
F^*_p(v)=\int f^*_p(t,v(t))\,d\tau(t).
\]
\end{lemma}
\begin{proof}
Let $q\in L^2(T,\mathcal{H},\tau)$. Then
\[
g_p(v(t),\log_pq(t))-f(t,q(t))\leq f^*_p(t,v(t))
\]
for any $t\in T$, and hence
\[
g_p(v,\log_pq)-F(q)\leq \int f^*_p(t,v(t))\,d\tau(t).
\]
Now, taking the supremum over all such $q$, we get
\[
\int f^*_p(t,v(t))\,d\tau(t)\geq \sup_{q\in L^2(T,\mathcal{H},\tau)}\left\{g_p(v,\log_pq)-F(q)\right\}.
\]
For the converse, let $\beta< \int f^*_p(t,v(t))\,d\tau(t)$. Now, we define a real-valued integrable function $\gamma$ such that $\beta<\int \gamma(t)\,d\tau(t)-\varepsilon$ for some $\varepsilon>0$ and $\gamma(t)<f^*_p(t,v(t))$ for all $t\in T$. If $\int f^*_p(t,v(t))\,d\tau(t)<+\infty$, then we set $\gamma(t):=f^*_p(t,v(t))-\varepsilon$ for $\varepsilon>0$ such that 
\[
\beta< \int f^*_p(t,v(t))\,d\tau(t)-\varepsilon.
\]
If $\int f^*_p(t,v(t))\,d\tau(t)=+\infty$, then let $N$ and $\varepsilon>0$ be such that 
\[
\beta< \int \min\{N,f^*_p(t,v(t))\}\,d\tau(t)-\varepsilon
\]
and set $\gamma(t):=\min\{N,f^*_p(t,v(t))\}-\varepsilon$. Define 
\[
\Gamma(t):=\{x\in \mathcal{H}\mid g_p(v(t),\log_px)-f(t,x)\geq\gamma(t)\}.
\]
Then $\Gamma$ is graph measurable by Lemma \ref{epiLem} and so, by completeness, measurable. Further, the values of $\Gamma$ are closed and non-empty, the latter since $\gamma(t)<f^*_p(t,v(t))$. Hence, by the Kuratowski–Ryll-Nardzewski measurable selection theorem (see Theorem 8.1.3 in \cite{AubinFrankowska2009}), there exists a measurable $s:T\to \mathcal{H}$ such that
\[
g_p(v(t),\log_ps(t))-f(t,s(t))\geq\gamma(t)\text{ for all }t\in T.
\]
Now, take $u_0\in\mathrm{dom}F$, which exists since $F$ is proper. Then, there exists an integrable $\alpha$ such that
\[
g_p(v(t),\log_pu_0(t))-f(t,u_0(t))\geq\alpha(t)\text{ for all }t\in T.
\]
Using Lemma \ref{Radon}, take now an increasing sequence $T_n$ such that $\tau(T\setminus \bigcup_nT_n)=0$ and such that $\overline{s(T_n)}$ is compact for all $n\in\mathbb{N}$. Define
\[
u_n(t):=\begin{cases}s(t)&\text{if }t\in T_n,\\u_0(t)&\text{if }t\in T\setminus T_n.
\end{cases}
\]
Clearly we have $u_n\in L^2(T,\mathcal{H},\tau)$. In particular, we have
\[
\int g_p(v(t),\log_pu_n(t))\,d\tau(t)-\int f(t,u_n(t))\,d\tau(t)\geq\int_{T_n}\gamma(t)\,d\tau(t)+\int_{T\setminus T_n}\alpha(t)\,d\tau(t).
\]
As $\int_{T\setminus T_n}\alpha(t)\,d\tau(t)\to 0$ as well as $\int_{T_n}\gamma(t)\,d\tau(t)\to \int_{T}\gamma(t)\,d\tau(t)$ when $n\to\infty$, there is an $n$ such that 
\[
\int_{T\setminus T_n}\alpha(t)\,d\tau(t)\geq -\varepsilon/2\text{ and } \int_{T_n}\gamma(t)\,d\tau(t)\geq \int\gamma(t)\,d\tau(t)-\varepsilon/2.
\]
Hence, for such an $n$, we get
\[
g_p(v,\log_pu_n)-F(u_n)=
\int g_p(v(t),\log_pu_n(t))\,d\tau(t)-\int f(t,u_n(t))\,d\tau(t)\geq\int\gamma(t)\,d\tau(t)-\varepsilon>\beta.
\]
This completes the proof.
\end{proof}

The above can now be used rather immediately to show the following characterization of the subdifferential of $F$:

\begin{lemma}\label{partialFcharac}
Let $T$ be separable and let $\mathcal{H}$ be a separable Hilbert-Hadamard space, let $f:T\times \mathcal{H}\to (-\infty,+\infty]$ be a proper normal convex integrand and suppose that $F$ is proper and lsc. Then, for any $x\in \mathcal{H}$:
\[
\partial F(x)=\{v\in L^2(T,T_x\mathcal{H},\tau)\mid v(t)\in\partial f(t,x)\text{ almost surely}\}.
\]
\end{lemma}
\begin{proof}
By Lemma \ref{conjugateProp}, we have $v\in\partial F(x)$ if, and only if, $F^*_x(v)+F(x)= 0$. By Lemma \ref{conjugateCharac}, the latter is equivalent to
\begin{align*}
0&=\int f^*_x(t,v(t))\,d\tau(t)+\int f(t,x)\,d\tau(t)\\
&=\int (f^*_x(t,v(t)) + f(t,x))\,d\tau(t).
\end{align*}
As $f^*_x(t,v(t)) + f(t,x)\geq 0$ by Lemma \ref{conjugateProp}, the above is hence is equivalent to $f^*_x(t,v(t)) + f(t,x)=0$ almost surely, which by Lemma \ref{conjugateProp} again is equivalent to $v(t)\in\partial f(t,x)$ almost surely.
\end{proof}

The following key lemma now requires the linearity of the tangent cones as well as the Lipschitz continuity of the integrand. 

\begin{lemma}\label{chainRule}
Let $T$ be separable and let $\mathcal{H}$ be a separable full Hilbert-Hadamard space. Let $f:T\times \mathcal{H}\to (-\infty,+\infty]$ be a proper normal convex integrand and suppose that $f$ is $L^2$-Lipschitz, that is
\[
\vert f(t,x)-f(t,y)\vert\leq L(t)d(x,y)\text{ for all }x,y\in \mathcal{H}\text{ and }t\in T,
\]
where $L\in L^2(T,(0,\infty),\tau)$ with $\int L(t)\,d\tau(t)>0$. Assume that $f(\cdot,u_0)$ is integrable for some $u_0\in \mathcal{H}$. Then, for any $x\in \mathcal{H}$:
\[
\partial \underline{f}(x)=\left\{\int v\,d\tau\mid v\in \partial F(x)\right\}.
\]
\end{lemma}

Conceptually, the above result is similar to (a special case of) Theorem I-29 in \cite{CastaingValadier1977} on the representation of the subdifferential $\partial (h\circ A)(x)$ of a composition $h\circ A$ of a convex function $h$ with a linear operator via its conjugate $A^*$ as $A^*(\partial h(A(x)))$.

To prove Lemma \ref{chainRule}, we rely on the directional derivative $D_x h(v)$ at a point $x\in \mathcal{H}$ and $v\in T_x\mathcal{H}$ of a convex and Lipschitz continuous function $h:\mathcal{H}\to\mathbb{R}$. To define such an object, we follow the approach of \cite{DiMarinoGigliPasqualettoSoultanis2021} and first define the function $\sigma_xh:\Sigma'_x\mathcal{H}\to \overline{\mathbb{R}}$ via
\[
\sigma_xh(\gamma):=\lim_{t\to 0}\frac{h(\gamma(t))-h(x)}{t}
\]
for $x\in \mathcal{H}$ and $\gamma\in\Sigma'_x\mathcal{H}$. Note that by convexity of $h$, the above fraction is actually nondecreasing in $t$.  A related object was studied already in \cite{MovahediBehmardiHosseini2015} as well as \cite{Lytchak2004,Petrunin2007,Plaut2002} (the latter in different, partly more general settings). As shown in \cite{DiMarinoGigliPasqualettoSoultanis2021}, using that $h$ is Lipschitz, this function can be extended to $T_x\mathcal{H}$ and this extension retains various properties of $h$.

\begin{lemma}[{Proposition 2.16 in \cite{DiMarinoGigliPasqualettoSoultanis2021}}]\label{dirDerProp}
Let $\mathcal{H}$ be a Hadamard space and let $h:\mathcal{H}\to\mathbb{R}$ be convex and Lipschitz continuous with constant $L>0$. Then, for each $x\in \mathcal{H}$, there exists a unique continuous map $D_xh:T_x\mathcal{H}\to\mathbb{R}$ such that
\[
D_xh(\gamma)=\sigma_x h(\gamma)\text{ for all }\gamma\in \Sigma'_x\mathcal{H}.
\]
Further, $D_xh$ is convex, Lipschitz continuous with constant $L$ and positively $1$-homogeneous, that is $D_xh(\lambda v)=\lambda D_xh(v)$ for any $v\in T_x\mathcal{H}$ and $\lambda\geq 0$.
\end{lemma}

As every convex and positively $1$-homogeneous function is sub-additive over linear spaces, we in particular get the following result for linear tangent cones:

\begin{corollary}\label{dirDerSublin}
Let $\mathcal{H}$ be a Hadamard space and let $h:\mathcal{H}\to\mathbb{R}$ be convex and Lipschitz continuous with constant $L>0$. Fix $x\in \mathcal{H}$. If $T_x\mathcal{H}$ is a Hilbert space, then $D_xh$ is sub-additive.
\end{corollary}

The last property we need of the directional derivative is that it characterizes the subgradient:

\begin{lemma}[{extending Theorem 2.7 in \cite{MovahediBehmardiHosseini2015}}]\label{dirDerSubgrad}
Let $\mathcal{H}$ be a Hadamard space and let $h:\mathcal{H}\to\mathbb{R}$ be convex and Lipschitz continuous with constant $L>0$. Then, for any $x\in \mathcal{H}$ and $v\in T_x\mathcal{H}$, $v\in\partial h(x)$ if, and only if, $g_x(v,w)\leq D_xh(w)$ for all $w\in T_x\mathcal{H}$.
\end{lemma}
\begin{proof}
Suppose that $v\in\partial h(x)$, that is
\[
h(y)\geq h(x)+g_x(v,\log_xy)\text{ for all }y\in \mathcal{H}.
\]
Hence, for $\gamma\in \Sigma'_x\mathcal{H}$ with length $l$, we have
\[
h(\gamma(t))\geq h(x)+g_x(v,\log_x\gamma(t))\text{ for all }t\in [0,l].
\]
Now, note that
\[
g_x(v,\log_x\gamma(t))= tg_x(v,\gamma).
\]
Dividing by $t$ and taking $t\to 0$ yields $\sigma_xh(\gamma)\geq g_x(v,\gamma)$. Using continuity of $D_xh$ and $g_x$, this now extends to $\gamma\in\Sigma_x\mathcal{H}$. For $w=\lambda\gamma\in T_x\mathcal{H}$, we then get
\[
D_xh(\lambda\gamma)=\lambda D_xh(\gamma)\geq \lambda g_x(v,\gamma)=g_x(v,\lambda\gamma)
\]
using positive $1$-homogeneity of $D_xh$ and $g_x$. 

For the converse, suppose $g_x(v,w)\leq D_xh(w)$ for all $w\in T_x\mathcal{H}$, that is in particular 
\[
g_x(v,\log_xy)\leq D_xh(\log_xy)=d(x,y)D_xh(\gamma_{x,y})\leq d(x,y) \frac{h(\gamma_{x,y}(t))-h(x)}{t}
\]
for $y\neq x$ and for all $t\in (0,d(x,y)]$. Taking $t=d(x,y)$, we get
\[
g_x(v,\log_xy)\leq h(y)-h(x)\text{ for all }y\in \mathcal{H}
\]
and so $v\in \partial h(x)$.
\end{proof}

We can now give the proof of Lemma \ref{chainRule}. Note for this that if $f$ is $L^2$-Lipschitz with $L\in L^2(T,(0,\infty),\tau)$, integrating the respective Lipschitz condition 
\[
\vert f(t,x(t))-f(t,y(t))\vert\leq L(t)d(x(t),y(t))\text{ for all }t\in T
\]
for $x,y\in L^2(T,\mathcal{H},\tau)$ and using the Cauchy-Schwarz inequality yields that
\[
\vert F(x)-F(y)\vert\leq \left(\int L^2(t)\,d\tau(t)\right)^{1/2}d_2(x,y),
\]
so that $F$ is Lipschitz on $L^2(T,\mathcal{H},\tau)$, and in particular $\mathrm{dom}F=L^2(T,\mathcal{H},\tau)$ if $F$ is proper, i.e.\ if there exists a $u_0\in L^2(T,\mathcal{H},\tau)$ such that $f(\cdot,u_0(\cdot))\in L^1(T,\mathbb{R},\tau)$. In particular, also $\underline{f}$ is Lipschitz on $\mathcal{H}$ and $\mathrm{dom}\underline{f}=\mathcal{H}$ if $\underline{f}$ is proper, i.e.\ if there exists a $u_0\in \mathcal{H}$ such that $f(\cdot,u_0)\in L^1(T,\mathbb{R},\tau)$, which holds if, and only if, $F$ is proper in this case.

Note also that for $x\in \mathcal{H}$ and $v\in T_x\mathcal{H}$, we have $D_x F(v)=D_x \underline{f}(v)$. To see this, note first that for $\gamma\in\Sigma'_x\mathcal{H}$, we have $\sigma_x\underline{f}(\gamma)=\sigma_x F(\gamma)$ simply because $\underline{f}$ and $F$ coincide on $\mathcal{H}$. Then, for the unique extensions $D_xF$ of $\sigma_xF$ on $T_xL^2(T,\mathcal{H},\tau)$ and $D_x\underline{f}$ of $\sigma_x\underline{f}$ on $T_x\mathcal{H}$, this extends to $\gamma\in\Sigma_x\mathcal{H}$ by continuity of $D_xF$ and $D_x\underline{f}$. Lastly, for $v=\lambda\gamma\in T_x\mathcal{H}$, we then have
\[
D_xF(v)=\lambda D_xF(\gamma)=\lambda D_x\underline{f}(\gamma)=D_x\underline{f}(v),
\]
using Lemma \ref{dirDerProp}.

\begin{proof}[Proof of Lemma \ref{chainRule}]
If $v\in\partial F(x)$, then by Lemma \ref{partialFcharac}, we have $v(t)\in\partial f(t,x)$ almost surely, that is
\[
f(t,y)\geq f(t,x)+g_x(v(t),\log_xy)\text{ for all }y\in \mathcal{H}.
\]
Integrating and using Lemma \ref{intEx} yields $\int v\,d\tau\in \partial \underline{f}(x)$. For the converse, let $u\in \partial \underline{f}(x)$ and define the function $l_0:T_x\mathcal{H}\to\mathbb{R}$ by $l_0(w):=g_x(u,w)$. Then Lemma \ref{dirDerSubgrad} yields
\[
l_0(w)=g_x(u,w)\leq D_x\underline{f}(w)=D_xF(w)
\]
for any $w\in T_x\mathcal{H}$. Since $T_x\mathcal{H}$ is a Hilbert space, so is $T_xL^2(T,\mathcal{H},\tau)=L^2(T,T_x\mathcal{H},\tau)$ of which $T_x\mathcal{H}$ is a subspace of. By Lemma \ref{dirDerProp} and Corollary \ref{dirDerSublin}, we hence have that $l_0$ is a continuous linear functional on $T_x\mathcal{H}$ which is dominated by the continuous and sublinear functional $D_xF:T_xL^2(T,\mathcal{H},\tau)\to\mathbb{R}$. By the Hahn-Banach theorem, $l_0$ extends to a continuous linear function $l:T_xL^2(T,\mathcal{H},\tau)\to\mathbb{R}$ such that $l(w)\leq D_xF(w)$ for all $w\in T_xL^2(T,\mathcal{H},\tau)$. By the Riesz representation theorem, there is a $v\in T_xL^2(T,\mathcal{H},\tau)$ such that
\[
l(w)=g_x(v,w)\text{ for all }w\in T_xL^2(T,\mathcal{H},\tau).
\]
Since we thus have $g_x(v,w)\leq D_xF(w)$ for all $w\in T_xL^2(T,\mathcal{H},\tau)$, we get $v\in\partial F(x)$. Further, for any $w\in T_x\mathcal{H}$, we have
\[
g_x(u,w)=l_0(w)=l(w)=g_x(v,w)=\int g_x(v(t),w)\,d\tau(t)=g_x\left(\int v\,d\tau,w\right)
\]
so that $\int v\,d\tau=u$.
\end{proof}

We now turn to the main result of this section, the converse to the above Lemma \ref{firstInterchange} over full Hilbert-Hadamard spaces.

\begin{theorem}\label{secondInterchange}
Let $\mathcal{H}$ be a separable full Hilbert-Hadamard space and $T$ be a complete and separable probability space. Let $f:T\times \mathcal{H}\to (-\infty,+\infty]$ be a proper normal convex integrand and suppose that $f$ is $L^2$-Lipschitz and that $f(\cdot,u_0)$ is integrable for some $u_0\in \mathcal{H}$. Then, for any $x\in \mathcal{H}$: 
\[
\int \partial f(t,x)\,d\tau(t)=\partial\left(\int f(t,\cdot)\,d\tau(t)\right)(x).
\]
\end{theorem}
\begin{proof}
By Lemma \ref{firstInterchange}, we only have to show the inclusion $\partial\underline{f}(x)\subseteq\underline{\partial f}(x)$. Let $u\in \partial\underline{f}(x)$. Hence, by Lemma \ref{chainRule}, we have that there exists a $v\in \partial F(x)$ such that $u=\int v\,d\tau$. As $v\in \partial F(x)$, Lemma \ref{partialFcharac} yields that $v(t)\in\partial f(t,x)$ almost surely. Hence, we have $u\in \underline{\partial f}(x)$.
\end{proof}

For the above Theorem \ref{secondInterchange}, we in particular highlight the following natural question:

\begin{question}
Are the assumptions of Theorem \ref{secondInterchange} that $f$ is $L^2$-Lipschitz or that $T$ is separable necessary?
\end{question}

\section{A probabilistic Lie-Trotter-Kato formula}\label{sec:LTK}

Combining Theorem \ref{secondInterchange} with Theorem \ref{L1LawVector}, we in particular obtain the following result on full Hilbert-Hadamard spaces, in analogy to Theorem 1 in \cite{Salim2023}. 

\begin{theorem}\label{HadCorr}
Let $(\Omega,\mathsf{F},\PP)$ and $(E,\mathsf{E},\mu)$ be probability spaces and let $\mathcal{H}$ be a separable full Hilbert-Hadamard space. Further, fix an i.i.d.\ sequence $(\xi_i)$ of random variables $\xi_i:\Omega\to E$ with distribution $\mu$. Let $A:E\to \mathcal{M}(\mathcal{H})$ be an integrable random monotone vector field such that $\frac{1}{n}\sum_{i=1}^n A(\xi_i)$ satisfies the surjectivity condition $\PP$-a.s., for all $n\in\mathbb{N}$. Then
\[
\frac{1}{n}\sum_{i=1}^n A(\xi_i)\to^R \int A(e)\, d\mu(e)\quad\PP\text{-a.s.}
\]
as $n\to\infty$.

In particular, suppose that $(E,\mathsf{E},\mu)$ is complete and separable and $A=\partial f$ for a proper normal convex integrand $f:E\times \mathcal{H}\to (-\infty,+\infty]$ which is $L^2$-Lipschitz and where $f(\cdot,u_0)$ is integrable for some $u_0\in \mathcal{H}$. Then
\[
\partial \left(\frac{1}{n}\sum_{i=1}^n f(\xi_i,\cdot)\right) \to^R \partial \left(\int f(e,\cdot)\, d\mu(e)\right)\quad\PP\text{-a.s.}
\]
as $n\to\infty$.
\end{theorem}

This latter conclusion can now be used to derive a probabilistic representation of the gradient flow associated with $\int f(e,\cdot)\, d\mu(e)$ which resembles the well-known Lie-Trotter-Kato formula. 

Concretely, let $\mathcal{H}$ be a Hadamard space and let $f:\mathcal{H}\to (-\infty,+\infty]$ be a convex lsc function. Then, given $t>0$, one can define the mapping $S_t:\overline{\mathrm{dom}f}\to \mathcal{H}$ by
\[
S_t(x):=\lim_{k\to\infty}\left(\mathrm{prox}_{\frac{t}{k}f}\right)^{(k)}(x),
\]
the so-called exponential formula. The limit exists and is uniform in $t$ on bounded subintervals of $(0,\infty)$ (see Theorem 1.13 in \cite{Mayer1998}), and the above map can be continuously extended to $t\in [0,\infty)$ by $S_0(x):=x$ (where we likewise set $\mathrm{prox}_{0f}(x):=x$). Moreover, the collection of maps $(S_t)$ is a nonexpansive semigroup, that is every $S_t$ is nonexpansive and $S_{t+s}=S_t\circ S_s$. As such, the above definition and existence result provides a nonlinear extension of a similar result of Crandall and Liggett \cite{CrandallLiggett1971} set over Hilbert spaces, where this semigroup can also be characterized as the solution semigroup of the parabolic inclusion problem $-\dot{x}\in \partial f(x)$ which generalizes the gradient flow equation $-\dot{x}=\nabla f(x)$ to nonsmooth convex functions $f$ (in fact, these results extend to general maximally monotone operators in place of $\partial f$). In particular, the nonlinear variant inherits many aspects of the associated theory of this inclusion, such as regularity of the flow or characterizations via an evolution variational inequality, and further extends to various other classes of spaces, making it a seminal tool in geometric analysis. We refer to  \cite{AmbrosioGigliSavare2008,Bacak2014a} for further background.

The first key result we need is that resolvent convergence actually implies the convergence of the gradient flow semigroup, as established by Ba\v{c}\'ak \cite{Bacak2015}:

\begin{lemma}[Theorem 4.4 in \cite{Bacak2015}]\label{resToSemi}
Let $\mathcal{H}$ be a Hadamard space and let $h_n:\mathcal{H}\to\mathbb{R}$ as well as $h:\mathcal{H}\to \mathbb{R}$ be convex and lsc. Let $S_t(x)$ be the gradient flow semigroup associated with $h$ and $S_{n,t}(x)$ be the gradient flow semigroup associated with $h_n$. If $\mathrm{prox}_{\lambda h_n}(x)\to\mathrm{prox}_{\lambda h}(x)$ as $n\to\infty$ for all $\lambda>0$ and $x\in \mathcal{H}$, then $S_{n,t}(x)\to S_{t}(x)$ as $n\to\infty$ for all $t\geq 0$ and $x\in \mathcal{H}$.
\end{lemma}

Fundamentally, there is the following theorem of Stojkovi\'c \cite{Stojkovic2012}, of which an easier proof is given by Ba\v{c}\'ak in \cite{Bacak2014c}.

\begin{lemma}[Theorem 4.4 in \cite{Stojkovic2012}]\label{LTK}
Let $\mathcal{H}$ be a Hadamard space and let $h:\mathcal{H}\to \mathbb{R}$ be such that $h=\sum_{k=1}^Nh_k$ where each $h_k:\mathcal{H}\to\mathbb{R}$ is convex and lsc. Let $S_t(x)$ be the gradient flow semigroup associated with $h$. Then, for every $t\geq 0$ and all $x\in \mathcal{H}$:
\[
S_t(x)=\lim_{n\to\infty} \left(\mathrm{prox}_{\frac{t}{n} h_N} \circ\dots\circ \mathrm{prox}_{\frac{t}{n} h_1}\right)^{(n)}(x).
\]
\end{lemma}

If the functions $h_1,\dots,h_N$ are suitably Lipschitz, we can improve the above to the following uniform quantitative version, essentially by following (and at points quantitatively improving) the argument provided in \cite{Stojkovic2012}:

\begin{lemma}\label{quantLTK}
Let $\mathcal{H}$ be a Hadamard space and let $h:\mathcal{H}\to \mathbb{R}$ be such that $h=\sum_{k=1}^Nh_k$ where each $h_k$ is convex and $L_k$-Lipschitz. Let $S_t(x)$ be the gradient flow semigroup associated with $h$. Then, for every $t\geq 0$ and all $x\in \mathcal{H}$ as well as any $n\in\mathbb{N}$:
\[
d\left(S_t(x),\left(\mathrm{prox}_{\frac{t}{n} h_N} \circ\dots\circ \mathrm{prox}_{\frac{t}{n} h_1}\right)^{(n)}(x)\right)\leq 8Ltn^{-1/4},
\]
where $L:=\sum_{k=1}^NL_k$.
\end{lemma}
\begin{proof}
Given $r>0$, write $Q_r:=\mathrm{prox}_{r h_N} \circ\dots\circ \mathrm{prox}_{r h_1}$. Note that $Q_r$ is nonexpansive. Consider the $(\lambda,r)$ resolvent $J_{\lambda,r}$ associated with $Q_r$, defined as in Definition 3.2 in \cite{Stojkovic2012}. For each $\lambda,r>0$, the map $J_{\lambda,r}$ is also nonexpansive (see Lemma 3.3 in \cite{Stojkovic2012}). Let $S^{r}_t(x)$ be the nonexpansive semigroup generated by $J_{\lambda,r}$ as in Theorem 3.10 in \cite{Stojkovic2012}, and note that one has
\[
d(S^{r}_t(x),Q_r^{(m)}(x))\leq 2\sqrt{\left(m-\frac{t}{r}\right)^2+\frac{t}{r}}\cdot d(x,Q_r(x))
\]
for any $t,r>0$, $x\in \mathcal{H}$ and $m\in\mathbb{N}$ (see Lemma 3.11 in \cite{Stojkovic2012}). Setting $r:=t/m$, we get
\[
d(S^{t/m}_t(x),Q_{t/m}^{(m)}(x))\leq 2\sqrt{m}d(x,Q_{t/m}(x))
\]
As in e.g.\ eq.\ (3.32) in \cite{Stojkovic2012} (see also \cite{AmbrosioGigliSavare2008}), we further have
\[
d\left((\mathrm{prox}_{\frac{t}{m} h})^{(m)}(x),S_t(x)\right)\leq\frac{t}{\sqrt{2}m}\vert\partial h\vert(x),
\]
where, again, $\vert\partial h\vert(x)$ is the metric slope of $h$ at $x$. As $h$ is $L$-Lipschitz, we have $\vert\partial h\vert(x)\leq L$. Using Theorem 3.1.6 in \cite{AmbrosioGigliSavare2008}, we get
\[
d(x,\mathrm{prox}_{r h_i}(x))\leq r\vert\partial h_i\vert(x)\leq r L_i
\]
and hence
\[
d(x,Q_r(x))\leq \sum_{k=1}^N d(\mathrm{prox}_{r h_k} \circ\dots\circ \mathrm{prox}_{r h_1}(x),\mathrm{prox}_{r h_{k-1}} \circ\dots\circ \mathrm{prox}_{r h_1}(x))\leq \sum_{k=1}^N r L_k=rL.
\]
Combined, we have
\[
d(S^{t/m}_t(x),Q^{(m)}_{t/m}(x))\leq 2tL/\sqrt{m}.
\]
Eq.\ (4.13) in \cite{Stojkovic2012} gives
\[
\frac{1}{2\lambda}d^2(v,J_{\lambda,r}(x))+\frac{1}{2\lambda}d^2(x,J_{\lambda,r}(x))+\sum_{i=1}^Nh_i(x_i)\leq h(v)+\frac{1}{2\lambda}d^2(x,v)
\]
for $x_i:=\mathrm{prox}_{r h_i}(x_{i-1})$ with $x_0:=J_{\lambda,r}(x)$. We get $d(x_i,x_{i-1})\leq r L_i$ again using Theorem 3.1.6 in \cite{AmbrosioGigliSavare2008}, so that $d(x_i,x_0)\leq r\sum_{k=1}^iL_k$ and hence
\[
h(x_0)-\sum_{i=1}^Nh_i(x_i)\leq \sum_{i=1}^N L_id(x_i,x_0)\leq r\sum_{i=1}^N L_i\sum_{k=1}^iL_k\leq rL^2,
\]
using that $h_i$ is $L_i$-Lipschitz. Hence
\[
\frac{1}{2\lambda}d^2(v,J_{\lambda,r}(x))-rL^2\leq \left(h(v)+\frac{1}{2\lambda}d^2(x,v)\right)-\left(h(J_{\lambda,r}(x))+\frac{1}{2\lambda}d^2(x,J_{\lambda,r}(x))\right).
\]
For $v=\mathrm{prox}_{\lambda h}x$, the right-hand side is nonpositive, so that
\[
\frac{1}{2\lambda}d^2(\mathrm{prox}_{\lambda h}(x),J_{\lambda,r}(x))\leq rL^2
\]
and hence $d(J_{\lambda,r}(x),\mathrm{prox}_{\lambda h}(x))\leq L\sqrt{2\lambda r}$. Using that both maps are nonexpansive, we get
\[
d(J_{\lambda,r}^{(n)}(x),(\mathrm{prox}_{\lambda h})^{(n)}(x))\leq n L\sqrt{2\lambda r}.
\]
Lastly, note that by Theorem 3.10 in \cite{Stojkovic2012}, we have
\[
d(S^r_t(x),J_{t/n,r}^{(n)}(x))\leq 2t\frac{1}{\sqrt{n}}\frac{d(x,Q_r(x))}{r}\leq \frac{2tL}{\sqrt{n}}.
\]
Combined, we get
\begin{align*}
d(S^r_t(x),S_t(x))&\leq d(S^r_t(x),J_{t/k,r}^{(k)}(x))+d(J_{t/k,r}^{(k)}(x),\mathrm{prox}_{t/k h}^{(k)}(x))+d(\mathrm{prox}_{t/k h}^{(k)}(x),S_t(x))\\
&\leq \frac{2tL}{\sqrt{k}}+L\sqrt{2trk}+\frac{t}{\sqrt{2}k}L.
\end{align*}
For $r=t/n$ and $k=\lceil \sqrt{n}\rceil$, we get
\[
d(S^{t/n}_t(x),S_t(x))\leq Lt\left(\frac{2}{n^{1/4}}+\frac{2}{\sqrt{n}}+\frac{1}{\sqrt{2}\sqrt{n}}\right)\leq 6Ltn^{-1/4}.
\]
All together, we finally obtain
\begin{align*}
d(Q^{(n)}_{t/n}(x),S_t(x))&\leq d(Q^{(n)}_{t/n}(x),S^{t/n}_t(x))+d(S^{t/n}_t(x),S_t(x))\\
&\leq 2tL n^{-1/2}+6Ltn^{-1/4}\leq 8Ltn^{-1/4},
\end{align*}
as claimed.
\end{proof}

Note that by the above Lemma \ref{quantLTK}, the convergence in the Lie-Trotter-Kato formula is in particular uniform in $t$ on each bounded time interval in this case (compare Problem 4.10 in \cite{Bacak2023}). 

Combined, we can now show the following:

\begin{theorem}\label{probLTK}
Let $(\Omega,\mathsf{F},\PP)$ and $(E,\mathsf{E},\mu)$ be probability spaces, with the latter complete and separable, and let $\mathcal{H}$ be a separable full Hilbert-Hadamard space. Further, fix an i.i.d.\ sequence $(\xi_i)$ of random variables $\xi_i:\Omega\to E$ with distribution $\mu$. Let $f:E\times \mathcal{H}\to (-\infty,+\infty]$ be a proper normal convex integrand which is $L^2$-Lipschitz and where $f(\cdot,u_0)$ is integrable for some $u_0\in \mathcal{H}$. Write $S_t(x)$ for the gradient flow associated with $\int f(e,\cdot)\, d\mu(e)$. Then, there exists a set $\Omega'\subseteq\Omega$ with $\PP(\Omega')=1$ such that for any $\omega\in\Omega'$ as well as for any $t\geq 0$ and $x\in \mathcal{H}$:
\[
S_t(x)=\lim_{n\to\infty} \left(\mathrm{prox}_{\frac{t}{n^2} f(\xi_n(\omega),\cdot)} \circ\dots\circ \mathrm{prox}_{\frac{t}{n^2} f(\xi_1(\omega),\cdot)}\right)^{(n)}(x).
\]
In particular, the above limit holds almost surely.
\end{theorem}
\begin{proof}
Using Theorem \ref{HadCorr}, we get that there exists a single set $\Omega_0$ of measure one such that
\[
\mathrm{prox}_{\lambda \frac{1}{n}\sum_{i=1}^n f(\xi_i(\omega),\cdot)}(x)\to\mathrm{prox}_{\lambda \int f(e,\cdot)\, d\mu(e)}(x)
\]
as $n\to\infty$, for all $\lambda>0$ and $x\in \mathcal{H}$ as well as $\omega\in\Omega_0$. Let $S^\omega_{n,t}(x)$ be the gradient flow semigroup associated with $\frac{1}{n}\sum_{i=1}^n f(\xi_i(\omega),\cdot)$. Using Lemma \ref{resToSemi}, we hence get that
\[
S^\omega_{n,t}(x)\to S_t(x)
\]
as $n\to\infty$, for all $t\geq 0$ and $x\in \mathcal{H}$ as well as $\omega\in \Omega_0$. Now, using Lemma \ref{quantLTK}, we get
\[
d\left(S^\omega_{n,t}(x),\left(\mathrm{prox}_{\frac{t}{jn} f(\xi_n(\omega),\cdot)} \circ\dots\circ \mathrm{prox}_{\frac{t}{jn} f(\xi_1(\omega),\cdot)}\right)^{(j)}(x)\right)\leq 8tj^{-1/4}\left(\frac{1}{n}\sum_{i=1}^nL(\xi_i(\omega))\right)
\]
for any $j,n\in\mathbb{N}$, $t\geq 0$ and $x\in \mathcal{H}$ as well as $\omega\in\Omega_0$. As $L$ is square-integrable and $(\xi_i)$ is i.i.d., we get $\frac{1}{n}\sum_{i=1}^nL(\xi_i)\to \int L(e)\,d\mu(e)$ $\PP$-a.s.\ as $n\to \infty$ by the strong law of large numbers, say on a set $\Omega_1$ of measure one. In particular, $L^s(\omega):=\sup_{n\in\mathbb{N}}\left(\frac{1}{n}\sum_{i=1}^nL(\xi_i(\omega))\right)<+\infty$ for all $\omega\in\Omega_1$. In particular, we have
\[
d\left(S_t(x),\left(\mathrm{prox}_{\frac{t}{n^2} f(\xi_n(\omega),\cdot)} \circ\dots\circ \mathrm{prox}_{\frac{t}{n^2} f(\xi_1(\omega),\cdot)}\right)^{(n)}(x)\right)\leq d(S_t(x),S^\omega_{n,t}(x))+8tn^{-1/4}L^s(\omega)
\]
for all $\omega\in\Omega_0\cap\Omega_1$ as well as $t\geq 0$, $x\in\mathcal{H}$ and $n\in\mathbb{N}$, and so we get the result by letting $n\to\infty$.
\end{proof}

As mentioned in the introduction, the use of Theorem \ref{HadCorr} can be avoided in certain cases which moreover allow for the simplification of the Lipschitz assumption in the above Theorem \ref{probLTK}.

At first, in the special case where $\mathcal{H}$ is a separable Hilbert space, we can employ a (consequence of a) strong law of large numbers for convex functions established by King and Wets \cite{KingWets1991}. 

\begin{lemma}[consequence of Theorem 2.4 in \cite{KingWets1991}]\label{KingWets}
Let $(\Omega,\mathsf{F},\PP)$ and $(E,\mathsf{E},\mu)$ be probability spaces, with the latter complete, and let $\mathcal{H}$ be a separable Hilbert space. Further, fix an i.i.d.\ sequence $(\xi_i)$ of random variables $\xi_i:\Omega\to E$ with distribution $\mu$. Let $f:E\times \mathcal{H}\to (-\infty,+\infty]$ be a normal convex integrand which is $L^1$-Lipschitz and where $f(\cdot,u_0)$ is integrable for some $u_0\in \mathcal{H}$. Then
\[
\frac{1}{n}\sum_{i=1}^n f(\xi_i,\cdot) \to^M \int f(e,\cdot)\, d\mu(e)\quad\PP\text{-a.s.}
\]
as $n\to\infty$, with $\to^M$ meaning convergence in the sense of Mosco (see e.g.\ \cite{Bacak2014a}).
\end{lemma}

Over Hilbert spaces, by a result of Attouch \cite{Attouch1984}, Mosco convergence of convex functions is (essentially) equivalent to the pointwise convergence of proximal maps. While it is not known whether this result fully transfers to Hadamard spaces (see Problem 3.14 in \cite{Bacak2023}), one direction remains valid which will be the crucial one for the present paper.

\begin{lemma}[Theorem 4.1 in \cite{Bacak2015}]\label{MoscoToRes}
Let $(\mathcal{H},d)$ be a Hadamard space and let $h_n:\mathcal{H}\to (-\infty,+\infty]$ be convex and lsc for each $n\in\mathbb{N}$. If $h_n\to^M h$ as $n\to\infty$, then $\mathrm{prox}_{\lambda h_n}(x)\to\mathrm{prox}_{\lambda h}(x)$ as $n\to\infty$ for all $\lambda>0$ and $x\in\mathcal{H}$.
\end{lemma}

This now allows us to establish the following variant of Theorem \ref{probLTK}:

\begin{theorem}\label{probLTKHilbert}
Let $(\Omega,\mathsf{F},\PP)$ and $(E,\mathsf{E},\mu)$ be probability spaces, with the latter complete, and let $\mathcal{H}$ be a separable Hilbert space. Further, fix an i.i.d.\ sequence $(\xi_i)$ of random variables $\xi_i:\Omega\to E$ with distribution $\mu$. Let $f:E\times \mathcal{H}\to (-\infty,+\infty]$ be a proper normal convex integrand which is $L^1$-Lipschitz and where $f(\cdot,u_0)$ is integrable for some $u_0\in \mathcal{H}$. Write $S_t(x)$ for the gradient flow associated with $\int f(e,\cdot)\, d\mu(e)$. Then, there exists a set $\Omega'\subseteq\Omega$ with $\PP(\Omega')=1$ such that for any $\omega\in\Omega'$ as well as for any $t\geq 0$ and $x\in \mathcal{H}$:
\[
S_t(x)=\lim_{n\to\infty} \left(\mathrm{prox}_{\frac{t}{n^2} f(\xi_n(\omega),\cdot)} \circ\dots\circ \mathrm{prox}_{\frac{t}{n^2} f(\xi_1(\omega),\cdot)}\right)^{(n)}(x).
\]
In particular, the above limit holds almost surely.
\end{theorem}
\begin{proof}
Proceed as in the proof of Theorem \ref{probLTK}, using Lemmas \ref{KingWets} and \ref{MoscoToRes} in place of Theorem \ref{HadCorr} to derive that there exists a single set $\Omega_0$ of measure one such that
\[
\mathrm{prox}_{\lambda \frac{1}{n}\sum_{i=1}^n f(\xi_i(\omega),\cdot)}(x)\to\mathrm{prox}_{\lambda \int f(e,\cdot)\, d\mu(e)}(x)
\]
as $n\to\infty$, for all $\lambda>0$ and $x\in \mathcal{H}$ as well as $\omega\in\Omega_0$. From that point onwards, the proof proceeds similar, noting that it suffices that $L$ is integrable to apply the strong law of large numbers.
\end{proof}

The second case we consider is that $\mathcal{H}$ is a locally compact Hadamard space. In that case, we cannot only weaken the $L^2$-Lipschitz assumption to $L^1$-Lipschitz continuity but also do not have to require any additional assumptions on the tangent cones of $\mathcal{H}$. The result that makes this possible is the following (consequence of a) strong law of large numbers for convex functions due to Artstein and Wets \cite{ArtsteinWets1995}:

\begin{lemma}[consequence of Theorem 2.3 in \cite{ArtsteinWets1995}]\label{ArtsteinWets}
Let $(\Omega,\mathsf{F},\PP)$ and $(E,\mathsf{E},\mu)$ be probability spaces, with the latter complete, and let $\mathcal{H}$ be a separable Hadamard space. Further, fix an i.i.d.\ sequence $(\xi_i)$ of random variables $\xi_i:\Omega\to E$ with distribution $\mu$. Let $f:E\times \mathcal{H}\to (-\infty,+\infty]$ be a normal convex integrand which is $L^1$-Lipschitz and where $f(\cdot,u_0)$ is integrable for some $u_0\in \mathcal{H}$. Then
\[
\frac{1}{n}\sum_{i=1}^n f(\xi_i,\cdot) \to^\Gamma \int f(e,\cdot)\, d\mu(e)\quad\PP\text{-a.s.}
\]
as $n\to\infty$, with $\to^\Gamma$ meaning $\Gamma$-convergence in the sense of De Giorgi (see e.g.\ \cite{Bacak2014a}).
\end{lemma}

The reason why we are restricted to locally compact Hadamard spaces when applying the above Lemma \ref{ArtsteinWets} is that it only establishes the weaker $\Gamma$-convergence, which is generally not enough to guarantee the convergence of the proximal maps (this already fails in infinite-dimensional Hilbert spaces, see e.g.\ Example 3.2 in \cite{LevyPoliquinThibault1995}). However, over locally compact Hadamard spaces, weak and strong convergence coincide (see e.g.\ Example 3.2 in \cite{LytchakPetrunin2023}), and so $\Gamma$- and Mosco-convergence also coincide. It is hence exactly outside of Hilbert spaces and locally compact Hadamard spaces where our Theorem \ref{HadCorr} has its full use.

We now end this paper with the corresponding variant of Theorem \ref{probLTK} in the case of locally compact Hadamard spaces:

\begin{theorem}\label{probLTKProper}
Let $(\Omega,\mathsf{F},\PP)$ and $(E,\mathsf{E},\mu)$ be probability spaces, with the latter complete, and let $\mathcal{H}$ be a locally compact Hadamard space. Further, fix an i.i.d.\ sequence $(\xi_i)$ of random variables $\xi_i:\Omega\to E$ with distribution $\mu$. Let $f:E\times \mathcal{H}\to (-\infty,+\infty]$ be a proper normal convex integrand which is $L^1$-Lipschitz and where $f(\cdot,u_0)$ is integrable for some $u_0\in \mathcal{H}$. Write $S_t(x)$ for the gradient flow associated with $\int f(e,\cdot)\, d\mu(e)$. Then, there exists a set $\Omega'\subseteq\Omega$ with $\PP(\Omega')=1$ such that for any $\omega\in\Omega'$ as well as for any $t\geq 0$ and $x\in \mathcal{H}$:
\[
S_t(x)=\lim_{n\to\infty} \left(\mathrm{prox}_{\frac{t}{n^2} f(\xi_n(\omega),\cdot)} \circ\dots\circ \mathrm{prox}_{\frac{t}{n^2} f(\xi_1(\omega),\cdot)}\right)^{(n)}(x).
\]
In particular, the above limit holds almost surely.
\end{theorem}
\begin{proof}
Proceed as in the proof of Theorem \ref{probLTK} (as well as Theorem \ref{probLTKHilbert}), now using Lemmas \ref{ArtsteinWets} and \ref{MoscoToRes} in place of Theorem \ref{HadCorr}.
\end{proof}

\noindent {\bf Acknowledgments:} I want to thank Morenikeji Neri and Thomas Powell for helpful comments on a previous draft of this paper. The author interacted with OpenAI's GPT-5.6 while working on this paper, as detailed in the following:
\begin{enumerate}
\item Examples \ref{exPPA} and \ref{ExNoHH} were directly provided by GPT-5.6.
\item Example \ref{HilHadEx} as well as Lemmas \ref{contRes}, \ref{Ameas}, \ref{HHtofullHH}, \ref{epiLem} and \ref{dirDerSubgrad} were first obtained by the author alone, partially in weaker form, and then have been extended and streamlined (together with their proofs) after interacting with GPT-5.6. GPT-5.6 made a similar contribution to Lemma \ref{almostPolish}, where it in particular suggested a much easier proof of Lemma \ref{resUnique} which does not require the surjectivity of $\log$, as was used in a first version of the author. However, the author then found another proof alone, the present one, which is quite different to this proof suggested by GPT-5.6.
\item The approach to Corollary \ref{L1LawHad} and Theorem \ref{L1LawVector} was first developed by the author alone over Hadamard manifolds without a focus on quantitative aspects, and then later substantially simplified after interactions with GPT-5.6, not yet being aware of the work by Yokota \cite{Yokota2018} (and which in particular was not referenced by GPT-5.6). The lift to Hadamard spaces was then carried out by the author alone, as was the quantitative analysis of this first purely qualitative proof which resulted in the non-asymptotic concentration inequalities provided in Theorem \ref{L1LawHadQuant} and Remark \ref{vectorLawQuant}. After learning of \cite{Yokota2018}, the resulting exposition was restructured.
\item The greatest mathematical contribution of GPT-5.6 is the formulation and proof of Lemma \ref{chainRule}. This result was first formulated (in a slightly different way) and proven over Hadamard manifolds by GPT-5.6 after various interactions on the content of Section \ref{interchange}. In particular, GPT-5.6 suggested the approach to use Hahn-Banach to extend a linear functional on the tangent cone. The author then developed the lift of this result to full Hilbert-Hadamard spaces, together with GPT-5.6.
\item The outline of the approach to Lemma \ref{quantLTK} was sketched first by GPT-5.6 after the author asked an adjacent question of removing a double limit in a first version of Theorem \ref{probLTK}, obtained by the author alone. It was later refined by the author alone.
\end{enumerate}
GPT-5.6 has also been used to find relevant related literature, minor improvements in arguments and typos. The author takes full accountability for the work. No parts of this paper were written or worded by an LLM.

\bibliographystyle{plain}
\bibliography{ref}

\end{document}